\documentclass[oneside]{amsart}

\usepackage{amsthm}
\usepackage{amsmath,amscd}
\usepackage{amssymb}
\usepackage{enumerate}
\usepackage{graphicx}
\usepackage[hidelinks,pagebackref,pdftex]{hyperref}
\usepackage{booktabs}
\usepackage{color}
\usepackage[dvipsnames]{xcolor}
\usepackage{import}
\usepackage{tikz-cd}
\usepackage{ulem}

\AtBeginDocument{%
   \def\MR#1{}
}

\usepackage{marginnote}
\long\def\@savemarbox#1#2{\global\setbox#1\vtop{\hsize\marginparwidth 
  \@parboxrestore\tiny\raggedright #2}}
\renewcommand*{\backref}[1]{}
\renewcommand*{\backrefalt}[4]{
  \ifcase #1
  [No citations.]
  \or [#2]
  \else [#2]
  \fi }

\numberwithin{equation}{section}
\theoremstyle{plain}
\newtheorem{theorem}[equation]{Theorem}
\newtheorem{corollary}[equation]{Corollary}
\newtheorem{lemma}[equation]{Lemma}

\newtheorem{proposition}[equation]{Proposition}

\newtheorem*{namedtheorem}{\theoremname}
\newcommand{\theoremname}{testing}

\theoremstyle{definition}
\newtheorem{definition}[equation]{Definition}
\newtheorem{remark}[equation]{Remark}

\newcommand{\calT}{{\mathcal{T}}}
\newcommand{\calS}{{\mathcal{S}}}

\newcommand{\fakeenv}{} 

\newenvironment{restate}[2]  
{
 \renewcommand{\fakeenv}{#2} 
 \theoremstyle{plain}
 \newtheorem*{\fakeenv}{#1~\ref{#2}} 
 \begin{\fakeenv}
}
{
 \end{\fakeenv}
}

\newtheorem{rthm}{Theorem}
 \theoremstyle{plain}
\newenvironment{reptheorem}[1]
  {\renewcommand\therthm{\ref*{#1}}\rthm}
  {\endrthm}

\DeclareMathOperator{\arcsinh}{arcsinh}

\makeatletter

\def\chaptermark#1{}

\def\chapter{%
  \if@openright\cleardoublepage\else\clearpage\fi
  \thispagestyle{plain}\global\@topnum\z@
  \@afterindenttrue \secdef\@chapter\@schapter}

\def\@chapter[#1]#2{\refstepcounter{chapter}%
  \ifnum\c@secnumdepth<\z@ \let\@secnumber\@empty
  \else \let\@secnumber\thechapter \fi
  \typeout{\chaptername\space\@secnumber}%
  \def\@toclevel{0}%
  \ifx\chaptername\appendixname \@tocwriteb\tocappendix{chapter}{#2}%
  \else \@tocwriteb\tocchapter{chapter}{#2}\fi
  \chaptermark{#1}%
  \addtocontents{lof}{\protect\addvspace{10\p@}}%
  \addtocontents{lot}{\protect\addvspace{10\p@}}%
  \@makechapterhead{#2}\@afterheading}
\def\@schapter#1{\typeout{#1}%
  \let\@secnumber\@empty
  \def\@toclevel{0}%
  \ifx\chaptername\appendixname \@tocwriteb\tocappendix{chapter}{#1}%
  \else \@tocwriteb\tocchapter{chapter}{#1}\fi
  \chaptermark{#1}%
  \addtocontents{lof}{\protect\addvspace{10\p@}}%
  \addtocontents{lot}{\protect\addvspace{10\p@}}%
  \@makeschapterhead{#1}\@afterheading}
\newcommand\chaptername{Chapter}

\def\@makechapterhead#1{\global\topskip 7.5pc\relax
  \begingroup
  \fontsize{\@xivpt}{18}\bfseries\centering
    \ifnum\c@secnumdepth>\m@ne
      \leavevmode \hskip-\leftskip
      \rlap{\vbox to\z@{\vss
          \centerline{\normalsize\mdseries
              \uppercase\@xp{\chaptername}\enspace\thechapter}
          \vskip 3pc}}\hskip\leftskip\fi
     #1\par \endgroup
  \skip@34\p@ \advance\skip@-\normalbaselineskip
  \vskip\skip@ }
\def\@makeschapterhead#1{\global\topskip 7.5pc\relax
  \begingroup
  \fontsize{\@xivpt}{18}\bfseries\centering
  #1\par \endgroup
  \skip@34\p@ \advance\skip@-\normalbaselineskip
  \vskip\skip@ }
\def\appendix{\par
  \c@chapter\z@ \c@section\z@
  \let\chaptername\appendixname
  \def\thechapter{\@Alph\c@chapter}}

\newcounter{chapter}

\newif\if@openright

\makeatother

\title[Polynomially many surfaces of fixed Euler characteristic]{Polynomially many surfaces of fixed Euler characteristic
in a hyperbolic 3-manifold}
\author{Marc Lackenby, Anastasiia Tsvietkova}
\date{}
\subjclass[2010]{}
\everymath{\displaystyle}

\begin{document}

\footnotesize
 \begin{abstract}
We give an upper bound for the number of compact essential orientable non-isotopic surfaces, with Euler characteristic at least some constant $\chi$, properly embedded in a finite-volume hyperbolic 3-manifold $M$, closed or cusped. This bound is a polynomial function of the volume of $M$, with degree that depends linearly on $|\chi|$.
 \end{abstract}

\maketitle
\normalsize

\section{Introduction}

Embedded incompressible surfaces play a central role in 3-manifold theory. A natural question is therefore how many essential properly embedded surfaces of a given Euler characteristic does a 3-manifold contain. This number is finite for hyperbolic 3-manifolds, but is infinite in general. It therefore is reasonable to restrict to the case of hyperbolic 3-manifolds.

This question about the number of surfaces in a 3-manifold has been studied in recent years. One can fix the 3-manifold and investigate how the surface count changes with surface genus or Euler characteristic. The earliest work in this direction is due to Haken, for embedded normal surfaces, and Kneser, for surfaces simultaneously embedded in a 3-manifold. Later a related question, about immersed $\pi_1$-injective surfaces, was studied by Masters \cite{Masters}, and Kahn and Markovic \cite{KahnMarkovic}. There, it is shown that in a closed hyperbolic 3-manifold, the number of such closed connected surfaces (up to homotopy) of genus $g$ grows like $g^{2g}$. This was recently extended to quasi-Fuchsian surface subgroups in cusped hyperbolic 3-manifolds in \cite{HRW}. (Calegari, Marques and Neves also counted quasi-Fuchsian surface subgroups in closed 3-manifolds in Section 4 of \cite{CMN} earlier.) For embedded surfaces (up to isotopy), the count can be smaller. In \cite{DGR}, Dunfield, Garoufalidis and Rubinstein proved that for a fixed hyperbolic 3-manifold (subject to some mild constraints, and without an embedded closed non-orientable essential surface), the count for closed embedded essential surfaces, possibly disconnected, is quasi-polynomial in Euler characteristic for all but finitely many of its values. In all these results, one needs to fix the 3-manifold to obtain the exact expression: the constants in \cite{Masters, KahnMarkovic} and the quasi-polynomials in \cite{DGR} depend on the 3-manifold.

If one instead fixes the Euler characteristic of the surfaces rather than fixing a 3-manifold, then upper bounds can be obtained that are \textit{universal} and \textit{polynomial}. By universal we mean that the expression and constants are given by an explicit general formula for all 3-manifolds.

\begin{theorem}\label{Thm:NumberOfSurfaces}
There are constants $c_1$ and $c_2$ with the following property.
Let $M$ be an orientable hyperbolic 3-manifold of finite volume $\mathrm{vol}(M)$, closed or with cusps. The number of properly embedded orientable essential surfaces in $M$ with Euler characteristic at least $\chi$, up to isotopy, is at most $(c_1 \mathrm{vol}(M))^{c_2 |\chi|}$.
\end{theorem}


Our theorem yields a simple bound for the number of essential surfaces in the exteriors of many links. Note that in cusped 3-manifolds, it applies to both closed surfaces and surfaces with non-empty boundary.

\begin{corollary}\label{Cor:Links}
There are constants $c'_1$ and $c_2$ with the following property.
Let $K$ be a hyperbolic link in the 3-sphere. Let 
its crossing number be $c(K)$. Then the number of properly embedded orientable essential
surfaces in the exterior of $K$ with 
Euler characteristic at least $\chi$, up to isotopy, is at most $(c'_1 c(K))^{c_2 |\chi|}$.
\end{corollary}

We will give the proof of this in Section \ref{Sec:ConclusionOfProof}. 

Note also that the restriction in the above results to orientable surfaces can be avoided.

\begin{corollary}\label{Cor:notorientable}

The statement of Theorem \ref{Thm:NumberOfSurfaces} extends to non-orientable surfaces that are $\pi_1$-injective and boundary-$\pi_1$-injective.
\end{corollary}

Here, a compact surface $S$ properly embedded in a 3-manifold $M$ is \emph{boundary-}\newline\emph{$\pi_1$-injective} if for any (possibly non-embedded) arc $\alpha$ in $S$ with endpoints in $\partial S$, the existence of a homotopy in $M$ of $\alpha$ into $\partial M$ implies the existence of a homotopy in $S$ of $\alpha$ into $\partial S$. The proof of Corollary \ref{Cor:notorientable} is also in Section \ref{Sec:ConclusionOfProof}.

\subsection{Related work}

The proof of our main result was inspired by the work of Hass, Thompson, and Tsvietkova in \cite{HTT1, HTT2} on prime alternating link complements in the 3-sphere and their Dehn fillings. Our Corollary \ref{Cor:Links} can be compared with two explicit and universal upper bounds they give: one for closed surfaces, and one for spanning surfaces, both being polynomial in terms of the link crossing number. In their proofs, surfaces are decomposed into pieces, and then an Euler characteristic argument is used to control the number of pieces. While we undertake similar steps, the techniques themselves are very different here, with the geometric ones being significantly more complex, which allows the higher level of generality. 

Another related work is that of Purcell and Tsvietkova \cite{PT1, PT2}. There, universal, explicit and polynomial upper bounds are established for a wide class of cusped 3-manifolds and their Dehn fillings. The cusped 3-manifolds are complements of certain links in compact, orientable,
irreducible 3-manifolds, possibly with boundary. The links are required to have a `nice' projection on some closed oriented embedded surface.  (`Nice' means alternating, prime in a certain natural sense, and checkerboard colourable.) The set of such 3-manifolds and the set of orientable hyperbolic 3-manifolds have a non-empty intersection, but none is a proper subset of the other. Our proof techniques are quite different, being much more geometric. In comparison, \cite{PT2}, as well as the paper \cite{PT1} that it relies upon, generalise topological methods from classical knot theory to achieve their bounds.

The techniques of our paper exploit the interaction between stable minimal surfaces and triangulations of 3-hyperbolic manifolds. In particular, 
we put such surfaces into normal position with respect to a `thick' triangulation while maintaining them as stable and minimal. Given the central place that normal surface theory and minimal surfaces have played in 3-manifold topology and geometry, it is reasonable to expect that the tools we develop will have other applications.

The interaction between minimal and normal surfaces was also recently explored by Appleboim \cite{Apple}, although he did not explore the counting question that is central to our paper, and the details of his method are different from ours. Among significant differences, Appleboim's surfaces are put into quasi-normal position (a relaxed version of normal position). We also strictly control the number of tetrahedra in a triangulation of the ambient 3-manifold: for this, we perturb our triangulations rather than iteratively subdivide them. Lastly, we consider 3-manifolds that are not necessarily compact.

The notion of a `thick' triangulation is important in this paper, and is defined below, following the work of Breslin \cite{Breslin2009, Breslin}, with some of Breslin's technical results and lemmas being refined here. Such triangulations were used, for example, for the classification of Heegaard splittings by Colding, Gabai, and Ketover \cite{CGK}, and studied in an algorithmic context by Edelsbrunner at al. \cite{Ed}.

\subsection{Acknowledgements}

M.L. was partially supported by the Engineering and Physical Sciences Research Council (grant number EP/Y004256/1). A.T. was partially supported by NSF CAREER grant DMS-2142487, by NSF research grants DMS-2005496, DMS-1664425 (formerly DMS-1406588), by Institute of Advanced Study under DMS-1926686 grant, and by Okinawa Institute of Science and Technology. A.T. would like to thank Jessica Purcell for drawing our attention to some related work concerning Delaunay and thick triangulations. For the purpose of open access, the authors have applied a CC BY public copyright licence to any author accepted manuscript arising from this submission.

\section{Overview of the proof}
\label{Sec:Overview}

The outline of the proof is as follows.

Let $M$ be a compact orientable 3-manifold with interior that admits a finite-volume hyperbolic structure. Our aim is to find an upper bound on the number of essential orientable surfaces $F$ properly embedded in $M$, up to isotopy, with Euler characteristic at least some constant $\chi$. This also yields a bound for surfaces of Euler characteristic exactly a constant.

According to the Margulis lemma \cite{KazhdanMargulis}, there is a universal constant $\epsilon_3 > 0$ with the following property. For any $\mu < \epsilon_3$, the interior of $M$ admits a decomposition into a `thick' part $M_{[\mu,\infty)}$ and a `thin' part $M_{(0,\mu)}$. Each component of the thin part is homeomorphic either to $T^2 \times (1,\infty)$ or to $S^1 \times \mathrm{int}(D^2)$. In the former case, this component is a neighbourhood of an end of $M - \partial M$, and is  called a `cusp'. In the latter case, the component is a regular neighbourhood of a geodesic with length less than $\mu$. At this stage, it may be useful to suppose that $M - \partial M$ has no geodesics with length less than $\mu$, so that the only components of $M_{(0,\mu)}$ are cusps. In that case, $M_{[\mu,\infty)}$ is homeomorphic to $M$.

A well known result of Thurston \cite{Thurston} states that  $M_{[\mu,\infty)}$ has a triangulation $T$ using at most $O(\mathrm{vol}(M))$ tetrahedra, where the implied constant depends on $\mu$. Since $F$ is essential, it can be isotoped into normal form with respect to $T$. This means that it intersects each tetrahedron in a collection of triangles and squares, as shown in Figure \ref{Fig:Normal}. For the basics of normal surface machinery, we refer the reader to Matveev's book \cite{Matveev}. The number of different normal disc types in $T$ is $7|T|$, where $|T|$ is the number of tetrahedra in $T$. Thus, $F$ is represented by a list of $7|T|$ non-negative integers, which count the number of normal discs of each type. This list of integers determines $F$ up to isotopy.

\begin{figure}
  \includegraphics[width=0.7\textwidth]{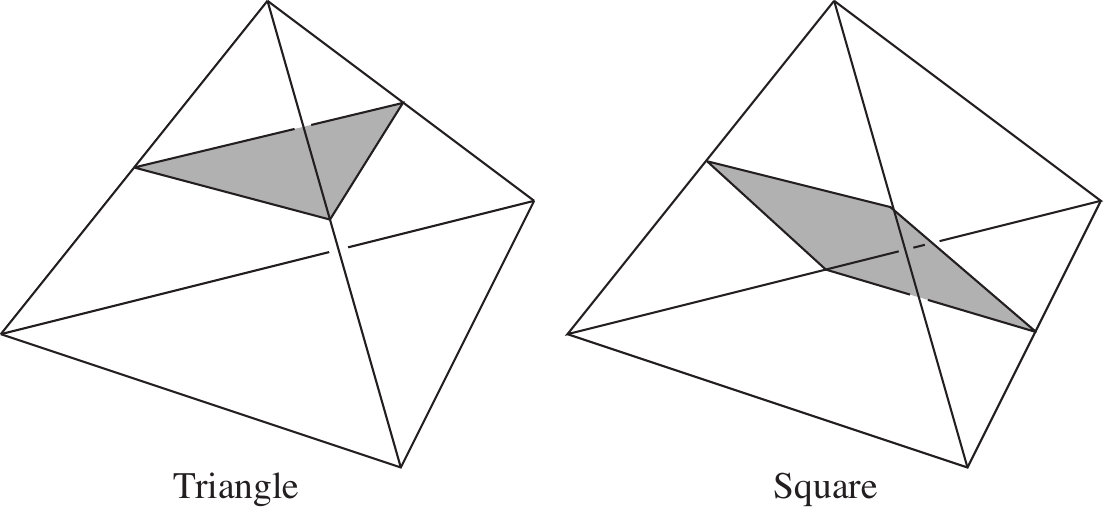}
  \caption{A normal disc in a tetrahedron is a triangle or square, as shown.}
  \label{Fig:Normal}
\end{figure}

\emph{Suppose} that we could arrange that the total number of normal discs of $F \cap T$ is at most $N |\chi(F)|$ for some universal constant $N$. Let us explain then how to give an upper bound for the number of surfaces.

To bound the number of possibilities for $F$ of Euler characteristic at least $\chi$, it would suffice to count the number of $(7|T|)$-tuples of non-negative integers that add to at most $N |\chi |$. Here, each integer represents the number of discs of some type in one of the $|T|$ tetrahedra.

We can use the `stars and bars' technique from combinatorics for $(7|T|+1)$-tuples, with the last integer representing the difference between $N |\chi|$ and the number of normal discs of $F \cap T$. Then the number of options for $F\cap T$ is 
\[ {N |\chi| + 7|T|   \choose  N |\chi|}.  \]
Since 
\[ {n \choose k}\leq n^k/k! \] 
we have an upper bound of 
\[ \frac{(N |\chi| +7|T| )^{N |\chi|}}{(N |\chi|)!}.\]
As $|T|$ is $O(\mathrm{vol}(M))$, we obtain an upper bound on the number of possibilities for $F$ that grows like 
$$\frac{\big(O(\mathrm{vol}(M)) + N|\chi|\big)^{N|\chi|}}{(N|\chi|)!}.$$

Alternatively, we may obtain a simpler bound on the number of surfaces, as follows. We can think of the normal discs of $F \cap T$ as labelled by integers from 1 to at most $N|\chi|$. For each such disc, we choose it to be one of at most $7|T|$ types. It is possible for different choices to give the same normal surface. It is also the case that many choices will not result in a normal surface. However, the number of choices is at most $O(\mathrm{vol}(M))^{N |\chi(F)|}$.

 This only controls the number of possibilities for the intersection between $F$ and $M_{[\mu,\infty)}$. However, the remainder $M_{(0,\mu)}$ has very simple topology, and so $F \cap M_{(0,\mu)}$ is nearly determined by the curves $F \cap \partial M_{[\mu,\infty)}$, and we have already controlled the number of possibilities for these curves. 

Thus, our main task is to bound the total number of normal discs of $F \cap T$ by a linear function of $|\chi(F)|$. Now, by results of Freedman, Hass and Scott \cite{FHS}, and Ruberman \cite{Ruberman}, $F$ can be isotoped to an embedded minimal surface (or one other possibility that we will discuss later). It then inherits a Riemannian metric with sectional curvature $\kappa$ that is everywhere at most $-1$, which is the sectional curvature of the metric on $M - \partial M$. By the Gauss-Bonnet theorem,
$$\mathrm{area}(F) = \int_F 1 \, dA \leq \int_F - \kappa \, dA = -2 \pi \chi(F),$$
where $dA$ is the area form on $F$.
So, the area of this representative for $F$ is bounded by a linear function of $|\chi(F)|$. It is therefore reasonable to try to establish that the number of normal discs of a normal representative for $F$ is also bounded by a linear function of $|\chi(F)|$. This is our main technical challenge.

We now give a more detailed outline of the remainder of the paper, section by section.

The triangulation that has a linear number of tetrahedra in terms of the manifold volume, as constructed by Thurston \cite{Thurston}, is unfortunately insufficient for our purposes. Although its number of tetrahedra is bounded above by a linear function of $\mathrm{vol}(M)$, it does not have extra properties that we require. Instead, we have to modify the triangulation by using a construction of Breslin \cite{Breslin}.

A tetrahedron $\Delta$ in hyperbolic 3-space is \emph{$(\theta, A, B)$-thick} for positive real numbers $\theta$, $A$ and $B$ if its dihedral angles are all at least $\theta$ and the length of each edge lies between $A$ and $B$. A triangulation of a hyperbolic 3-manifold is \emph{$(\theta, A, B)$-thick} if each tetrahedron is isometric to a $(\theta, A, B)$-thick tetrahedron in $\mathbb{H}^3$. Breslin proved that 
any hyperbolic 3-manifold $M$ has a triangulation with property that the tetrahedra lying in $M_{[\mu/2, \infty)}$ are $(\theta, A, B)$-thick. Here, $\mu$ is any positive real number less than the Margulis constant. In Breslin's result, $\theta(\mu)$, $A(\mu)$, $B(\mu)$ depend only on $\mu$, not on the manifold $M$.

Unfortunately, this also is not quite enough for our purposes, for two reasons. Firstly, we would like to work with a triangulation not of $M$ but of $M_{[\mu/2, \infty)}$. However, the boundary of $M_{[\mu/2, \infty)}$ is not piecewise totally geodesic and so when this boundary is non-empty, it is impossible to triangulate $M_{[\mu/2, \infty)}$ with geodesic hyperbolic tetrahedra. Furthermore, the boundary of $M_{[\mu/2, \infty)}$ may have very small injectivity radius, when $M$ contains a geodesic with length only a little less that $\mu/2$. Therefore, we work with a `fat' part of $M$, denoted $M^{\textrm{fat}}_{\mu/2}$, which is obtained from $M_{[\mu/2, \infty)}$ by attaching components of $M_{(0,\mu/2)}$ that are regular neighbourhoods of geodesics with length between $\mu/4$ and $\mu/2$. And we seek to triangulate not $M^{\textrm{fat}}_{\mu/2}$ but a manifold isotopic to it.

Secondly, we also need to know how $\theta(\mu)$, $A(\mu)$, $B(\mu)$ depend on $\mu$. We will prove that $\theta(\mu)$ can be taken to be a constant $\theta$, independent of $\mu$, and that $A(\mu)$ and $B(\mu)$ can be chosen to be $a\mu$ and $b\mu$ for universal constants $a$ and $b$. This is discussed in Section \ref{Sec:Triangulate}. The proof is a simple but technical adaption of Breslin's argument, and so is relegated to an appendix.

We therefore fix the constants $\theta$, $a$ and $b$ at this stage. However, we will defer the choice of $\mu$ to later in the paper. Thus, at this stage, we have not fixed our triangulation $T$.

As mentioned above, our goal is isotope the essential orientable properly embedded surface $F$ to a minimal surface, in order to apply the Gauss-Bonnet theorem. The fact that this is usually possible is due to Freedman, Hass and Scott \cite{FHS} and Ruberman \cite{Ruberman}. We discuss this material in Section \ref{Stable}. One alternative conclusion is that, instead of isotoping $F$ to a minimal surface, the final surface is the boundary of a thin regular neighbourhood of an embedded non-orientable minimal surface $F'$. A further small homotopy takes $F$ to an immersed minimal surface that double covers $F'$.

In our argument, it will be important that $F$ can be realised as a \emph{stable} minimal surface. The definition of this term is discussed in Section \ref{Stable}. Then a theorem of Schoen \cite{Schoen} (also see Colding and Minicozzi \cite{CMbook}) gives that there is a universal upper bound $C$ to the principal curvatures of a stable minimal surface in a hyperbolic 3-manifold. As a result, a `small' disc around any point of $F$ looks `almost totally geodesic'. In Section \ref{Sec:ConsequencesStability}, we make this notion precise with the definition of having $(\eta,\delta)$-almost constant normals. Here, $\eta$ is the scale that we are working at, and $\delta$ is the amount that these normals `vary'. We have to be careful with the latter notion, since we are comparing vectors at different points in $M$. A consequence of the theorem of Schoen is that for every $\delta > 0$, there is an $\eta > 0$ such that near any point $x$ of $F$, the surface has $(\eta, \delta)$-almost constant normals, provided $\eta$ is less than the injectivity radius of $M$ at $x$. Importantly, $\eta$ depends only on $\delta$ and not on the manifold $M$ or surface $F$.

Our next step is to build a triangulation of $M$ that is `sufficiently transverse' to $F$. We make this notion precise in Section \ref{Sec:TriangulationTransverseToSurface}. It has several conditions, but the main one is that the edges of the tetrahedra are not too close to being tangent to $F$. Again, it requires some care to make this precise, since we are comparing vectors based at different points of $M$. But the main conclusion is that when a surface is `almost totally geodesic' and is sufficiently transverse to a triangulation, then it intersects each tetrahedron in normal discs. The way that the triangulation is built is by forming a barycentric subdivision $T'$ of the triangulation $T$ constructed in Section \ref{Sec:Triangulate}, and then perturbing it, keeping its edges and faces totally geodesic. In Section \ref{Sec:Barycentric}, we show that there is always such a perturbation that can be made to $T'$, provided the scale $\mu$ that we choose is sufficiently small. We therefore now fix $\mu > 0$ so that this holds. Note that although we have perturbed $T'$, this triangulation is still combinatorially isomorphic to the barycentric subdivision of $T$. 

We show that this perturbed barycentric subdivision is $(\theta''', a'''\mu, b'''\mu)$-thick for universal constants $\theta'''$, $a'''$ and $b'''$. This implies that near each point of $F$ in $M^{\textrm{fat}}_{\mu/2}$, there is a normal disc with a uniform lower bound on its area. Hence, the number of normal discs in $M^{\textrm{fat}}_{\mu/2}$ is bounded above by a linear function of $\mathrm{area}(F) \leq -2 \pi \chi(F)$. The argument given at the beginning of this section then establishes the required upper bound on the number of possibilities for $F \cap M^{\textrm{fat}}_{\mu/2}$ up to isotopy. 

In Section \ref{Sec:ConclusionOfProof}, we complete the proof by also considering how $F$ lies in the thin part of $M$.

\section{Triangulating a fat part of the 3-manifold}\label{Sec:Triangulate}

\subsection{Thick-thin decomposition.}

Let $M$ be a Riemannian manifold and let $\mu >0$. The \emph{$\mu$-thin part} of $M$ is the subset of points $x\in M$ where the injectivity radius of $M$ at $x$ is less than $\mu/2$, usually denoted $M_{(0,\mu)}$. The thick part is the complement of the thin part, usually denoted $M_{[\mu, \infty)}$. There is a decomposition into a disjoint union $M=M_{(0,\mu)}\cup M_{[\mu,\infty)}$.

When $M$ is an orientable hyperbolic $n$-manifold of finite volume, the structure of the components of the thin part is well-understood due to the Margulis lemma \cite{KazhdanMargulis}. This asserts that there is a universal constant $\epsilon_n > 0$, depending only on the dimension $n$, with the property that for any $\mu <\epsilon_n$, the $\mu$-thin part has two sorts of components:
\begin{enumerate}
\item Cusps: these are the unbounded components, diffeomorphic to a flat $(n-1)$-manifold times a line.
\item Margulis tubes: these are neighbourhoods of closed geodesics of length $<\mu$ in $M$. They are bounded and diffeomorphic to a circle times a $(n-1)$-disc.
\end{enumerate}
In particular, an orientable complete finite-volume hyperbolic manifold is always diffeomorphic to the interior of a compact manifold with (possibly empty) boundary. For an orientable finite-volume hyperbolic 3-manifold, the thin part is therefore a disjoint union of solid tori and cusps.

A value $\epsilon_n > 0$ with the property that the $\epsilon_n$-thin part always has components only of the above form is known as an \emph{($n$-dimensional) Margulis constant}. We now fix a 3-dimensional Margulis constant $\epsilon_3$. It is known for instance that $0.104$ is such a constant, by Theorem 2 in the work of Meyerhoff \cite{Meyerhoff}. 
It is also known that $\epsilon_3 \leq 0.766$ (see \cite[Theorem 1.5(2)]{FPS}).

Although it is crucial to our arguments that there is a universal 3-dimensional Margulis constant $\epsilon_3$ that applies to all finite-volume hyperbolic 3-manifolds, we will nevertheless need to work at scales possibly much smaller than $\epsilon_3$. Thus, we will consider positive $\mu < \epsilon_3$ and will analyse $M_{(0,\mu)}$ and $M_{[\mu,\infty)}$. Thus we will speak of \emph{a} `thin' or `thick' part of $M$.

\subsection{A fat part of $M$}
Here and further in the paper, we assume $M$ is a complete hyperbolic 3-manifold with finite volume. 

\begin{definition}
Let $\mu > 0$ be a real number less than a 3-dimensional Margulis constant.
So $M_{(0,\mu/2)}$ consists of cusps and solid toral neighbourhoods of geodesics
with length less than $\mu/2$. Let $\textrm{Thin}_{\geq \mu/4}$ denote the union of those components
of $M_{(0,\mu/2)}$ that are regular neighbourhoods of geodesics with length at least $\mu/4$.
Define $M^{\textrm{fat}}_{\mu/2}$ to be $M_{[\mu/2,\infty)} \cup \textrm{Thin}_{\geq \mu/4}$.
\end{definition}

Thus, $M^{\textrm{fat}}_{\mu/2}$ is obtained from $M_{[\mu/2,\infty)}$ by attaching some
components of the thin part. These are regular neighbourhoods of geodesics that are not
too short: with length at least $\mu/4$. Note also that $\partial M^{\textrm{fat}}_{\mu/2}$ is a union of
certain components of $\partial M_{[\mu/2, \infty)}$.

\subsection{Triangulating a fat part of $M$ with a `thick' triangulation.}
Our aim in this subsection is to describe a triangulation of a fat part of $M$ with various properties. One of these properties is as follows.

\begin{definition}
Let $Y$ be a 3-dimensional submanifold of a hyperbolic 3-manifold $M$. A triangulation of $Y$ is 
\emph{geodesic} if each of its 3-simplices is isometric to a tetrahedron in hyperbolic 3-space.
\end{definition}

\begin{definition}
Let $Y$ be a 3-dimensional submanifold of a hyperbolic 3-manifold $M$. A triangulation of $Y$ is 
\emph{$(\theta, A,B)$-thick} triangulation, for positive real numbers $\theta$, $A$ and $B$, if it is geodesic and 
each of its tetrahedra has dihedral angles bounded below by $\theta$ and its edge lengths are in $[A, B]$.
\end{definition}

\begin{theorem}[Theorem 2 from \cite{Breslin} by Breslin] 
\label{BreslinTh}
Let $\mu$ be a positive real number less than a 3-dimensional Margulis constant. There exist positive constants $A := A(\mu), B := B(\mu)$, and $\theta := \theta(\mu)$ with the following property. Any complete hyperbolic 3-manifold $M$ has a geodesic triangulation such that each of its tetrahedra lying in 
the thick part $M_{[\mu/2, \infty)}$ is $(\theta, A,B)$-thick.
\end{theorem}

We will need a refined version of this result, which works at arbitrarily small scales (but focuses on finite-volume hyperbolic manifolds).

\begin{theorem}[Extension of Breslin's theorem] 
\label{Thm:ExtensionBreslin}
There are real numbers $0 < a \leq b \leq 1/40$ and $\theta > 0$ with the following property. Let $\mu$ be any positive real number less than the 3-dimensional Margulis constant, and let $M$ be a complete finite-volume hyperbolic 3-manifold. Then there is a  $(\theta, a\mu,b\mu)$-thick triangulation of a subset of $M$ 
isotopic to $M^{\textrm{fat}}_{\mu/2}$ and lying in $M_{[\mu/4,\infty)}$.
\end{theorem}

Since the proof of this result is mostly a technical analysis and adjustment of Breslin's argument, we defer it to the Appendix.

 \section{Stable minimal surfaces}\label{Stable}

\subsection{Background.}

 For basic background on surfaces in 3-space and their curvature, we refer the reader to, for example, Section 3 of \cite{Tu}.

Recall that $M - \partial M$ is a complete orientable hyperbolic 3-manifold with finite volume. In general, if $F$ is a surface in a Riemannian 3-manifold $N$, then an immersion $f$ (in our case, an embedding) of $F$ in $N$ induces a Riemannian metric on $F$, by pulling back the quadratic form on each tangent space. The area of $f$ is defined to be the area of $F$ in this induced metric. A surface $F$ smoothly embedded or immersed in a Riemannian
manifold $M$ is 
\emph{minimal} if it has mean curvature zero at all points.  Equivalently, for any compactly supported variation of $f$ \cite[Section 1.3]{CM}, the rate of change of the area of $f$ is zero.

Minimal surfaces do not necessarily minimise total area in their homotopy or isotopy classes. Rather, minimal surfaces are in general only critical points
for the total area function.  A surface is a least area
surface in a class of surfaces if it has finite area which realises the infimum of all
possible total areas for surfaces in this class. A surface which is of least area in a reasonable class of surfaces must be
minimal. The converse is false. It can also be proved that minimal surfaces are locally of least area, in the sense
that a variation supported on a sufficiently small region of the surface will not decrease the area.

Informally, a minimal surface is 
\emph{stable} if the second
variation is non-negative. For a precise and self-contained definition of a stable minimal surface, see for example \cite{Meeks}, where the surfaces considered are the stable $H$-surfaces (for $H$=0 an $H$-surface is the same as a minimal surface). A minimal and least area surface is stable. 

Recall that a map is \emph{proper} if inverse images of compact subsets are compact.

The main existence result that we will use is as follows. We will outline the results that prove this statement in the following subsection.

\begin{theorem}
\label{Thm:IsotopicToStableMinimal}
Let $M$ be a compact orientable 3-manifold whose interior admits a complete hyperbolic metric of finite volume. Let $F$ be a connected properly embedded essential orientable surface in $M$ other than a sphere or disc. Then one of the following holds:
\begin{enumerate}
\item $\mathrm{int}(F)$ is properly isotopic in $\mathrm{int}(M)$ to an embedded surface that has least area in its homotopy class, and which is therefore stable minimal;
\item $\mathrm{int}(F)$ is properly isotopic in $\mathrm{int}(M)$ to the boundary of a regular neighbourhood of a least area non-orientable surface $G$
 that is properly embedded in the interior of $M$; furthermore, there is a homotopy taking $\mathrm{int}(F)$ to an immersed surface that double covers $G$ and that has least area in its homotopy class.
\end{enumerate}
\end{theorem}

In the second case, we say that $F$ is the \emph{orientable double cover of $G$}. Note that in this case, both $\mathrm{int}(F)$ and $G$ are stable and minimal. We also say that $F$ is \emph{nearly isotopic} to the double cover of $G$.

\subsection{A survey of relevant results} We now collate the results that are needed to establish Theorem \ref{Thm:IsotopicToStableMinimal}.

The following result asserts the existence of a least area surface in a homotopy class. In the case when $M$ is closed, this is a consequence of Theorem 1.1 by Freedman, Hass, Scott \cite{FHS}, but the general case is Corollary 3.12 by Ruberman \cite{Ruberman}.

\begin{theorem}
\label{ExistenceLeastArea} Let $M$ be a compact orientable 3-manifold with interior admitting a finite-volume hyperbolic metric and let $F$
be a compact connected orientable surface other than a sphere, disc or torus. If $g \colon (F,\partial F) \rightarrow (M, \partial M)$ is an incompressible boundary-incompressible map, then
there is a least area map $f \colon \mathrm{int}(F) \rightarrow \mathrm{int}(M)$ which is properly homotopic to $g$.
\end{theorem}


The following result guarantees that the resulting least area surface is either embedded or the orientable double cover of an embedded surface.
When the surface is closed, this is a consequence of Theorem 5.1 by Freedman, Hass, Scott \cite{FHS}, but the general case is 
Theorem 3.16 by Ruberman \cite{Ruberman}.

\begin{theorem}
Let $M$ be a compact orientable $3$-manifold with interior admitting a finite-volume hyperbolic metric and let $F$
be a compact connected orientable surface. Let $f \colon \mathrm{int}(F) \rightarrow \mathrm{int}(M)$ be a least area map
that is proper, incompressible and boundary-incompressible, such that $f$ is properly homotopic to an embedding. Then either
\begin{enumerate}
\item $f$ is an embedding, or
\item $f$ is the orientable double cover of an embedded minimal non-orientable surface. 
\end{enumerate}
\end{theorem}

As a consequence of the above theorems, the initial embedded surface can be properly homotoped either to an embedded least area surface or to the orientable double cover of an embedded least area non-orientable surface. The following result implies that the homotopy can be achieved by an isotopy. 

\begin{theorem}
Let $M$ be a compact orientable irreducible 3-manifold with incompressible boundary. Let $F$ and $F'$ be properly embedded essential surfaces such $\mathrm{int}(F)$ and $\mathrm{int}(F')$ are properly homotopic in $\mathrm{int}(M)$. Then $F$ and $F'$ are ambient isotopic.
\end{theorem}

When $F$ and $F'$ are closed, this is a consequence of a theorem of Waldhausen \cite[Corollary 5.5]{Waldhausen}. When $F$ and $F'$ have non-empty boundary, the result is due to Johannson \cite[Proposition 19.1]{Johannson}.

 \section{Consequences of minimality}\label{ConsequencesMinimality}

We therefore now focus on a connected stable minimal surface, that may be orientable or non-orientable. For convenience, we refer to this surface as $F$. However, in the case where it is non-orientable, we are really interested in its orientable double cover.

\begin{lemma}\label{BallArea} Let $x$ be a point in a $\pi_1$-injective minimal surface $F$ embedded in the hyperbolic 3-manifold $M$. The surface $F$ inherits a Riemannian metric from $M$. Let $r$ be any real number less than or equal to the injectivity radius of $x$ in $M$. Denote an open disc around $x$ of radius $r>0$ on $F$ in this metric by $B_{F}(x,r)$. Then the area of $B_F(x,r)$ is at least $\pi r^2$.
\end{lemma}

\begin{proof} Recall that for a surface $F$ in a Riemannian 3-manifold $M$, the intrinsic Gaussian curvature at any point $x$ of $F$ is $K_i=K_e+K_M$, where $K_M$ is the sectional curvature of $M$, and $K_e$ is the extrinsic curvature of the surface. Here $K_M$ is equal to $-1$, since $M$ is hyperbolic. The curvature $K_e$ is the product of the two principal curvatures. Since $F$ is minimal, the principal curvatures have opposite signs or are both zero, and hence $K_e$ is non-positive. Therefore, $K_M+K_e \leq -1$.

Consider the exponential map $\exp: T_x F \rightarrow F $, where $T_x F$ is the tangent space to $F$ at the point $x$. We can endow $T_x F$ with the natural Euclidean metric arising from the induced Riemannian metric on $F$. For a surface of sectional curvature at most $-1$, an exponential map does not decrease distances. To prove this, one can use, for example, Hinge Comparison theorem (see, for example, \cite{CheegerEbin} or \cite{Petersen}). 

The exponential map is also a homeomorphism onto its image, when restricted to the open ball of radius $r$ about the origin in $T_x F$, for the following reason. If the exponential map was not injective, then there would be a geodesic $\gamma$ in $F$, based at $x$, starting and ending at $x$, and with length less than $2r$. Since the sectional curvature of $F$ is negative, such a geodesic would represent a non-trivial element of $\pi_1(F,x)$ \cite[Theorem 6.9.1]{Jost}.
 We are assuming that $F$ is $\pi_1$-injective. Therefore, $\gamma$ represents a non-trivial element of $\pi_1(M,x)$. However, $r$ is at most the injectivity radius of $M$, and therefore any curve based at $x$ with length less than $2r$ is homotopically trivial in $M$, which is a contradiction.

The area of a ball of radius $r$ in $T_x F$ is $\pi r^2$. Since the exponential map does not decrease distances and is a homeomorphism onto its image when restricted to the ball of radius $r$ about the origin, we deduce that the area of a ball of radius $r$ does not decrease under this map, i.e. $\pi r^2 = \mathrm{area}(B_{{\mathbb{E}}^2}(0, r))\leq \mathrm{area}(B_{F}(x, r))$. \end{proof}

By virtue of the Gauss-Bonnet Theorem, one can bound the area of the surface $F$ above by $-2\pi \chi(F)$, where $\chi(F)$ is the Euler characteristic of the surface.

\begin{lemma}\label{MuCover} Let $\eta$ be any positive real number. Let $F$ be a $\pi_1$-injective minimal surface in the hyperbolic 3-manifold $M$. Then there is a collection of at most $8 |\chi(F)|/ {\eta}^2$ discs in $F$ with radius $\eta$ that cover $F \cap M_{[\eta,\infty)}$. Here the distance $\eta$ is taken using the induced Riemannian metric on $F$ in $M$.
\end{lemma}

\begin{proof}

Choose a maximal collection of points ${x_1, ..., x_k}$ in $F \cap M_{[\eta,\infty)}$ such that $x_i, x_j$ are at least $\eta$-away from each other for each $i \not= j$. By maximality, any other point of $F \cap M_{[\eta,\infty)}$ is less than $\eta$ from $x_i$ for some $i$, and we can cover $F \cap M_{[\eta,\infty)}$ by discs of radius $\eta$ centred at $x_1, ..., x_k$. At the same time, $\eta/2$-discs centred at $x_1, ..., x_k$ are disjoint.

The resulting upper bound on $k$, the number of discs, is then the upper bound for the total area of $F$, which is $2\pi |\chi(F)|$, divided by the lower bound on the area of disc of radius $\eta/2$ in $F$, which is at least $\pi (\eta/2)^2$ by Lemma \ref{BallArea}. So we have at most $2\pi |\chi(F)|/(\pi (\eta/2)^2)=8 |\chi(F)|/{\eta}^2$ discs. \end{proof}

Note that the bound in the above lemma is linear in the Euler characteristic of the surface.

\section{Consequences of stability}
\label{Sec:ConsequencesStability}

The principal curvatures for 2-sided stable minimal embedded surfaces in a hyperbolic 3-manifold are bounded above by a universal constant $C$. This follows from the following theorem.

A surface has \emph{a trivial normal bundle} if there is a global orthonormal basis for the normal bundle. When the surface is immersed in an orientable 3-manifold (as in our case), this is equivalent to the surface being orientable.

For a surface $F$, denote by $A$ its second fundamental form. Recall that the second fundamental form is a bilinear form, represented by a 2 by 2 matrix $A$ as follows. If we have a point $x \in F$ and vectors $X, Y$ in $T_xF$, then $A(X, Y)= - \langle \nabla_XN , Y \rangle$, where $N$ is a field of unit normal vectors. (See subsection 1.2 in \cite{CMbook}.) It is easy to see that $A$ is symmetric. The principal curvatures are the eigenvalues of this symmetric matrix.

\begin{theorem}[Theorem 3 in \cite{Schoen} by Schoen, also see Theorem 2.10 in the textbook by Colding and Minicozzi \cite{CMbook}] 
\label{Schoen}
For each real number $k$, there exists $\eta > 0$ with the following property.
Suppose we have an immersed stable minimal surface $F$ in $M$ with a trivial normal bundle. Let $r_0\in(0,\eta]$, and for a point $x$ in $F$, let $B_F(x,r_0)$ be in $F - \partial F$. Suppose that the sectional curvatures $K_M$ of $M$ satisfy $|K_M|\leq k^2$. Then there exists a constant $C'(k)$ depending only on $k$ such that for all $0<\sigma\leq r_0$, $sup_{B_F(x, r_0-\sigma)}|A|^2\leq C'\sigma^{-2}$.\end{theorem}

The theorem is formulated so that for every $x$ and $r_0$ such a constant $C'$ exists. Schoen notes that the constant $C'$ depends only on the curvature of the ball $B_M(x, r_0)$ in \cite{Schoen}.
In particular $C'$ does not depend on $F$. For a hyperbolic manifold $M$, $K_M=-1$, so once $x$ and $r_0$ are fixed, the constant $C'$ is the same for every hyperbolic 3-manifold and surface. 

\begin{corollary} 
\label{Cor:UniformBoundCurvatures}
There is a universal constant $C$ such that the principal curvatures of any immersed stable minimal surface $F$ with trivial normal bundle in a complete hyperbolic 3-manifold $M$ are at most $C$.
\end{corollary}

\begin{proof}
Let $x$ be any point of $F$. We can set $r_0=\eta$, and $\sigma=\eta/2$. \end{proof}

We now want to examine the consequences of the above corollary. It implies that on a sufficiently small scale, the surface $F$ is `almost totally geodesic'. We make this precise in the following definition, where we compare the normal vectors of the surface at nearby points. In order to make this comparison, we use a basepoint $x$ in $M$, and we consider the exponential map $\exp_x \colon T_x M \rightarrow M$. The surface $F$ pulls back under $\exp_x^{-1}$ to a surface in $T_xM$. Since $T_xM$ is a Euclidean space, we can compare vectors based at different points of $T_xM$. In particular, if they are non-zero vectors, then there is a well-defined angle between these vectors, obtained through parallel transport of one of the vectors.

\begin{definition}\label{almost constant normals}
Let $F$ be a surface smoothly immersed in a Riemannian manifold $M$. Let $x$ be a point in $M$. Let $\eta$ and $\delta$ be positive real numbers. Then we say that $F$ has \emph{$(\eta, \delta)$-almost constant normals near $x$} if for every pair of normal vectors to the surface $\exp_x^{-1}(F) \cap B(0,\eta)$, the angle between them is at most $\delta$ or at least $\pi-\delta$. Here, $\exp_x \colon T_xM \rightarrow M$ is the exponential map, and $B(0,\eta)$ is the ball of radius $\eta$ about the origin in $T_xM$, with $T_xM$ given its canonical Euclidean metric. Since $T_xM$ is Euclidean, we may speak of the angle between two non-zero vectors, even when they are based at distinct points of $T_xM$. 
\end{definition}

The following statement is well known, but will be a key fact for us later on. So we give the result, and a brief sketch of proof.

\begin{corollary}
\label{Cor:AlmostConstantNormalComponent}
For any $\delta > 0$, there is an $\eta \in (0,\epsilon_3)$ with the following property. For any immersed stable minimal surface $F$ with trivial normal bundle in a complete hyperbolic 3-manifold $M$ and any point $x$ in $M$, the angle between any two normal vectors to any component of $\exp_x^{-1}(F) \cap B(0,\eta)$ is at most $\delta$ or at least $\pi - \delta$.
\end{corollary}

\begin{proof}

Fix a point $\tilde x$ in $\mathbb{H}^3$. Let $p \colon \mathbb{H}^3 \rightarrow M$ be the universal covering map. We may arrange that $p(\tilde x) = x$. Then there is a commutative diagram (where $(Dp)_{\tilde x}$ denotes the derivative of $p$ at $\tilde x$):
$$\begin{CD} 
T_{\tilde x} \mathbb{H}^3 @>{\exp_{\tilde x}}>> \mathbb{H}^3\\ 
@VV{(Dp)_{\tilde x}}V @VVpV\\
T_{x} M  @>{\exp_{x}}>> M \end{CD}$$
The vertical maps are local isometries.
Hence, the exponential map at $x$ is modelled on a single map $\exp_{\tilde x}$ that is independent of $x$ and $M$, up to isometry. 
So, when restricted to $B(0,\eta)$, there is a bound on the size of the iterated covariant derivates of $\exp_{x}$ and of its inverse. This bound depends
only the order of the derivative, not on $M$, $x$ and $\eta < \epsilon_3$. Hence
Corollary \ref{Cor:UniformBoundCurvatures} implies that the surface $\tilde F = \exp_x^{-1}(F) \cap B(0,\eta) \subset T_xM$ has a uniform bound on its principal curvatures.
Then by Lemma 2.4 in \cite{CMbook}, provided $\eta$ is sufficiently small, each component of $\tilde F$ is graphical, in the following sense. For any point $y$ in $\tilde F$, the component of $\tilde F$ containing $y$ is a graph of a function $u$ over $T_y \tilde F$. Moreover, by Lemma 2.4 in \cite{CMbook} and the remarks following it, $|\nabla u|$ is at most a constant times $\eta$. Therefore, the graph of $u$ is nearly a graph of a linear function. So, by picking $\eta$ sufficiently small, we can ensure that the angle between the normal at $y$ and the normal at any point of the component of $\tilde F$ containing $y$ is at most $\delta$ or at least $\pi - \delta$.  \end{proof}

When $\eta$ is at least the injectivity radius at $x$, the conclusion of the above corollary is that each component of $F \cap B_M(x,\eta)$ has $(\eta, \delta)$-almost constant normals near $x$.

A similar but easier proof (which is omitted) gives the following result.

\begin{lemma}
\label{Cor:AlmostConstantNormalPlane}
For any $\delta > 0$, there is an $\eta \in (0,\epsilon_3)$ with the following property. For any ball $B$ in hyperbolic 3-space with radius at most $\eta$ and for any disc $D$ properly embedded in $B$ lying in a totally geodesic hyperplane, $D$ has $(\eta, \delta)$-almost constant normals near any point of $B$.
\end{lemma}

\begin{proposition}
\label{Prop:AlmostConstantNormalsMultipleComponents}
Let $F$ be a surface smoothly and properly embedded in a Riemannian manifold $M$. Let $x$ be a point in $M$. Let $\eta$ and $\delta$ be positive real numbers. 
Suppose that the angle between any two normal vectors to any component of $\exp_x^{-1}(F) \cap B(0,\eta)$ is at most $\delta$ or at least $\pi - \delta$.
Then for any $\lambda \in (0,1)$, $F \cap B_M(x, \lambda \eta)$ has 
$(\lambda \eta, \delta')$-almost constant normals near $x$ for $\delta' = 4\delta + 2\sin^{-1}(\lambda)$.
\end{proposition}

\begin{proof}
The definition of almost constant normals near $x$ is phrased in terms of the exponential map $\exp_x$. Specifically, we are considering the normals to the surface $\tilde F = \exp_x^{-1}(F)$ in the ball $B(0,\eta)$ in the Euclidean space $T_xM$. Note that $\tilde F$ is smoothly embedded in $B(0,\eta) \cap T_xM$. So, it suffices to consider surfaces in this Euclidean space.

We may suppose that $\delta' < \pi/2$, as otherwise there is nothing to prove. Hence, $\delta < \pi/2$.
Let $F_1$ and $F_2$ be components of $\tilde F \cap B(0, \eta)$ that intersect $B(0, \lambda \eta)$. Let $p_1$ and $p_2$ be points in $F_1$ and $F_2$ that are closest to the origin in the Euclidean metric. We will show that the normal vectors $n_1$ and $n_2$ to $\tilde F$ at $p_1$ and $p_2$ have angle at most $\delta'' = 2 \delta + 2\sin^{-1}(\lambda)$ or at least $\pi - \delta''$. Because $F_1$ and $F_2$ have $(\eta,\delta)$-almost constant normals, this will imply that the angle between any two normal vectors for $F_1$ and $F_2$ is at most $\delta' = 2\delta + \delta''$ or at least $\pi - \delta'$, as required. Observe that, since $\delta' < \pi/2$, we have that $\pi/2 - \delta''/2 > \delta$.

Note that because $p_i$ is a closest point on $F_i$ to $0$, the normal vector $n_i$ is parallel to the vector $\overrightarrow{0p_i}$, unless $p_i$ is the origin. 

We claim that $\partial F_i$ is non-empty. If not, then $F_i$ is a closed surface embedded in the interior of $B(0,\eta)$. But, then by considering a supporting hyperplane with a given normal vector, we see that it has two unit normals that are perpendicular to each other. This contradicts the assumption that the angle between any two normals to $F_i$ is at most $\delta < \pi/2$ or at least $\pi - \delta$.

Let $q_i$ be any point on $\partial F_i = F_i \cap \partial B(0,\eta)$. We claim that the angle between the vector $\overrightarrow{p_iq_i}$ and $n_i$ lies between $\pi/2 - \delta$ and $\pi/2 + \delta$. In the case where $p_i$ is not the origin, let $P$ be the plane through $0$, $p_i$ and $q_i$. When $p_i$ is the origin, let $P$
be the plane containing $n_i$ and $\overrightarrow{0q_i}$. In both cases, $P$ contains $n_i$.
Assign orthogonal $(x,y)$ co-ordinates to this plane so that $n_i$ points in the $y$ direction. (See Figure \ref{Fig:NearlyConstantNormals}.)
Since $F_i$ has $(\eta, \delta)$-nearly constant normals near $0$ and $\delta < \pi/2$, no normal vector for $F_i$ is perpendicular to $n_i$. Hence, $F_i \cap P$ is the graph of a smooth function $f$ defined on an interval in $(-\eta, \eta)$.  Applying the Mean Value Theorem to this function, there is an $x$ in the domain of $f$ such that the tangent line to this graph at $(x,f(x))$ has the same slope as the line joining $p_i$ and $q_i$. This tangent line lies in a tangent plane to $F_i$, and hence as $F_i$ has $(\eta, \delta)$-nearly constant normals, the angle between this tangent line and $n_i$ lies between $\pi/2 - \delta$ and $\pi/2 + \delta$. This proves the claim.

\begin{figure}
  \includegraphics[width=0.48\textwidth]{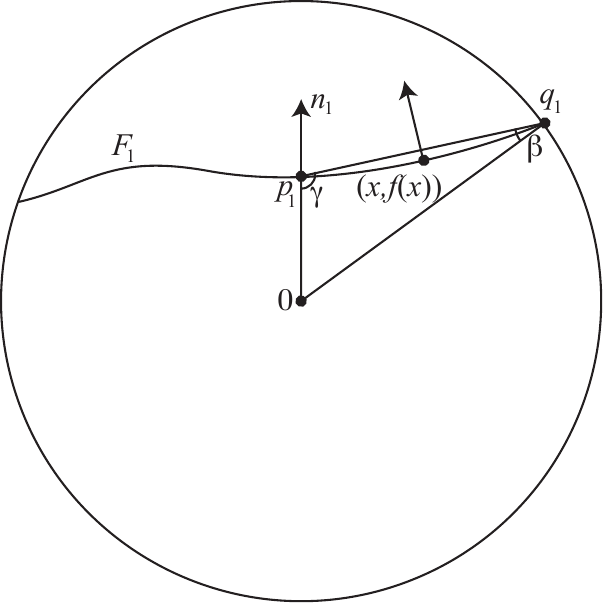}
  \caption{Shown is the intersection between $B(0, \eta)$ and the plane $P$ through $0$, $p_1$ and $q_1$.}
  \label{Fig:NearlyConstantNormals}
\end{figure}

Let $A_i$ (for $i=1,2$) be the annular strip in $\partial B(0,\eta)$ consisting of the points $q$ such that the angle between $n_i$ (translated to the origin) and the vector $\overrightarrow{0q}$ is between $\pi/2 - \delta''/2$ and $\pi/2 + \delta''/2$. 
We claim that $\partial F_i = F_i \cap \partial B(0,\eta)$ lies in $A_i$.  When $p_i =0$, this immediately follows from the previous claim. So we may assume that $p_i \not= 0$.
Recall that $q_i$ is a point on $\partial F_i$. View $0p_iq_i$ as a triangle in the plane $P$. Let $\beta$ and $\gamma$ be the internal angles at the vertices $q_i$ and $p_i$ of the triangle. Applying the sine rule, we have
$$|\sin \beta| = \frac{|\overrightarrow{0p_i}|}{|\overrightarrow{0q_i}|} |\sin \gamma| \leq  \frac{|\overrightarrow{0p_i}|}{|\overrightarrow{0q_i}|} \leq \lambda,$$
since $p_i \in B(0, \lambda \eta)$ and $q_i \in \partial B(0,\eta)$. 
So, $\beta \leq \sin^{-1}(\lambda)$. The angle $p_i0q_i$ equals $\pi - (\gamma + \beta)$. By the previous claim, $\pi - \gamma$ lies between $\pi/2 - \delta$ and $\pi/2 + \delta$. Hence, the angle $p_i0q_i$ lies between $\pi/2 - \delta - \sin^{-1}(\lambda)$ and $\pi/2 + \delta + \sin^{-1}(\lambda)$. The angle $p_i 0 q_i$ is the angle between $n_i$ and $\overrightarrow{0q_i}$. Thus, $\partial F_i$  does lie in $A_i$, as claimed.

We claim that each component of $\partial F_i$ has winding number one around the annulus $A_i$. Since $F_i$ is embedded in $M$ and $\partial F_i$ lies in $A_i$, the winding number is at most 1. We rule out the winding number being 0 as follows. The annulus $A_i$ is foliated by intervals, each of which is an arc of intersection with a plane through the origin containing $n_i$. If some component of $\partial F_i$ did not have winding number one around $A_i$, then, since $\partial F_i$ lies in $A_i$, $\partial F_i$ would be tangent to one of these intervals at some point $q$. Hence, the normal to $F_i$ at $q$ would lie in the plane perpendicular to the tangent at $q$ to this interval. This would imply that the angle between this normal and $n_i$ lies between $\pi/2 - \delta''/2$ and $\pi/2 + \delta''/2$. This contradicts the assumption that this angle is at most $\delta$ or at least $\pi - \delta$, since, as observed in the beginning of the proof, $\pi/2 - \delta''/2 > \delta$.

\begin{figure}
  \includegraphics[width=0.45\textwidth]{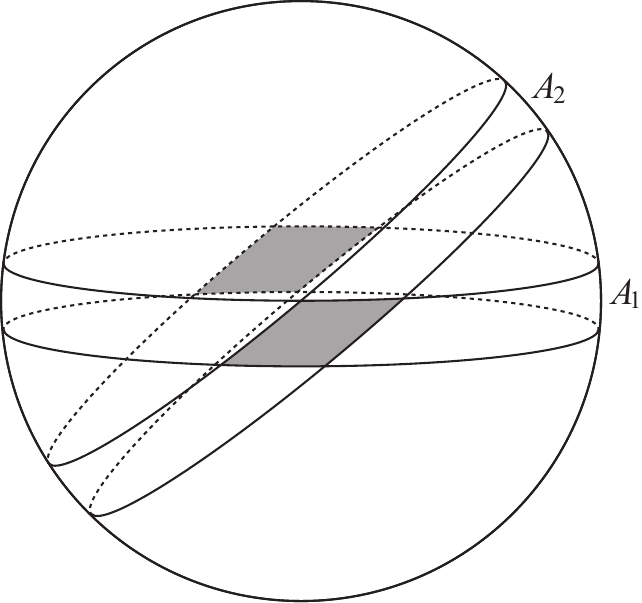}
  \caption{The annuli $A_1$ and $A_2$ in $\partial B(0,\eta)$. If the angle between $n_1$ and $n_2$ were greater than $\delta''$, then $A_1 \cap A_2$  consists of two discs.}
  \label{Fig:TwoAnnularStrips}
\end{figure}

 If the angle between $n_1$ and $n_2$ were greater than $\delta''$ or less than $\pi - \delta''$, 
 then $A_1 \cap A_2$ would consist of two discs. (See Figure \ref{Fig:TwoAnnularStrips}.)
 For each such disc $D$, $\partial D$ would consist of a concatenation of four arcs, one lying in a component of $\partial A_1$, the next lying in a component of $\partial A_2$, the next lying in the other component of $\partial A_1$, and the next lying in the other component of $\partial A_2$. Since $\partial F_1$ has winding number one around $A_1$, and $\partial F_2$ has winding number one around $A_2$, we deduce that $\partial F_1$ and $\partial F_2$ must intersect in $D$. This contradicts the fact that $F_1$ and $F_2$ are distinct components of the embedded surface $\tilde F$.

Thus, we deduce that the angle between $n_1$ and $n_2$ is at most $\delta''$ or at least $\pi - \delta''$, as required. \end{proof}

\begin{corollary}
\label{Cor:AlmostConstantNormalsSimplified}
For any $0 <\delta < \pi/2$, there is an $\eta >0$ with the following property. For any properly embedded stable minimal surface $F$ with trivial normal bundle in a complete hyperbolic 3-manifold $M$ and any point $x$ in $M$, $F$ has $(2\delta \eta/\pi, 6\delta)$-almost constant normals near $x$.
\end{corollary}

\begin{proof} By Corollary \ref{Cor:AlmostConstantNormalComponent}, the angle between any two normal vectors to any component of $\exp_x^{-1}(F) \cap B(0,\eta)$ is at most $\delta$ or at least $\pi - \delta$. Set $\lambda = \sin \delta$. Then by Proposition \ref{Prop:AlmostConstantNormalsMultipleComponents}, $F$ has $(\lambda \eta, \delta')$-almost constant normals near $x$ for $\delta' = 4\delta + 2\sin^{-1}(\lambda)$. Hence, it has $(2\delta \eta/\pi, 6\delta)$-almost constant normals near $x$, since $\lambda = \sin \delta \geq (2/\pi)\delta$. (The last inequality holds since  $0 <\delta < \pi/2$.) \end{proof}

\section{Triangulations transverse to a surface}
\label{Sec:TriangulationTransverseToSurface}

Recall that Theorem \ref{Thm:ExtensionBreslin} provides a triangulation $T$ of a subset of $M$ 
isotopic to $M^{\textrm{fat}}_{\mu/2}$ and lying in $M_{[\mu/4,\infty)}$. This triangulation is  $(\theta, a\mu,b\mu)$-thick.
In the next section, we will modify $T$ so that it becomes `transverse' to our given minimal surface $F$. But we first need to make the notion of transversality precise since we are working with surfaces that have geometric properties, and will need more than just topological transversality for them. We will do this by considering each tetrahedron $\Delta$ of the triangulation and by considering some point $x$ in $\Delta$. We will consider various objects near $x$, such as the surface $F$ or edges and faces of $\Delta$. We will want to consider the angles between various vectors associated with these objects, such as a tangent vector to $F$ at some point of $F$ or a tangent vector to some edge of $\Delta$ at a point of the edge. These vectors will typically not be based at the same point of $M$. In order to compare them, we will use the exponential map $\exp_x \colon T_xM \rightarrow M$. We will pull back the vectors $v_1$ and $v_2$ that we are considering to vectors based at points of $T_xM$, by using $\exp_x^{-1}$. Since $T_xM$ is a Euclidean space, we can then unambiguously refer to the angle in 3-space between these two vectors, despite the fact that they are based at different points of $T_xM$. Some care must be taken, though, because the exponential map $\exp_x$ is not an isometry. Hence, it need not send a unit normal to $\exp_x^{-1}(F)$ at some point to a unit normal to $F$.

We introduce the following notation. When a point $x$ in $M$ has been fixed and $\eta$ is less than the injectivity radius of $M$ at $x$, then for any subset $A$ of $M$, we let $\tilde A$ denote $\exp_x^{-1}(A) \cap B(0, \eta)$. Note that $\exp_x$ induces a diffeomorphism from $B(0,\eta)$ to $B_M(x, \eta)$, and this sends $\tilde A$ onto $A \cap B_M(x,\eta)$.

\begin{definition}
Let $\Delta$ be a geodesic tetrahedron in a hyperbolic 3-manifold $M$. Let $x$ be a point in $\Delta$. Let $\eta> 0$ be less than the injectivity radius of $M$ at $x$ and let $0 < \delta < \pi/2$. We say that a unit vector $u$ in $T_xM$ is \emph{$(\eta,\delta)$-nonperpendicular to the edges of $\Delta$} if, for each point $y$ in any edge $e$ of $\Delta$ with distance less than $\eta$ from $x$, the angle between $\tilde e$ at $\tilde y$ and $u$ is less than $\pi/2 - \delta$ or more than $\pi/2+\delta$.

\end{definition}

\begin{definition}\label{SufficientlyTransverse}
Let $F$ be a surface smoothly embedded in $M$, and let $\Delta$ be a geodesic tetrahedron embedded in $M$. We say that $F$ is \emph{sufficiently transverse} to $\Delta$ at scale $\eta >0$ if it is disjoint from the vertices of $\Delta$ and either $F \cap \Delta = \emptyset$ or there is some $x$ in $\Delta$ and some 
$\delta \in (0,\pi/6)$ such that
\begin{enumerate}
\item $\Delta$ lies in the ball of radius $\eta/4$ about $x$;
\item $F$ has $(\eta, \delta/5)$-almost constant normals near $x$ as in Definition \ref{almost constant normals}; 
\item a unit normal vector $v_{\tilde F}$ at some point of $\tilde F \cap \tilde \Delta$ is $(\eta, 2\delta)$-nonperpendicular to the edges of $\Delta$; 
\item $\eta$ is sufficiently small so that the angle between the unit normals at any two points of $\tilde R$, for any face $R$ of $\Delta$, is at most $\delta$;
\item $\eta$ is sufficiently small so that the angle between the unit tangent vectors at any two points of $\tilde e$, for any edge $e$ of $\Delta$, is at most $\delta$. 
\end{enumerate}

\end{definition}

When a surface is sufficiently transverse to each tetrahedron in a triangulation, we have the following important consequence.

\begin{proposition}\label{NormalDisk}
Let $\Delta$ be a geodesic tetrahedron in $M$. Suppose that  $F$ is a 2-sided surface properly embedded in $M$ that is sufficiently transverse to $\Delta$ at scale $\eta$. Then $F$ intersects $\Delta$ in normal discs. Moreover, for any point $p$ in $F$, the ball of radius $\eta/2$ around $p$ in $F$ (with its Riemannian metric) intersects at most one of these normal discs. 
\end{proposition}

\begin{proof} 
Fix the point $x \in \Delta$ and let the unit normal vector $v_{\tilde F}$ be as in Definition \ref{SufficientlyTransverse}. 

As assumed in Definition \ref{SufficientlyTransverse}, $F$ misses the vertices of $\Delta$. It also intersects the edges of $\Delta$ transversely, since otherwise there would be some point $y$ in the intersection between $F$ and some edge $e$ for which a unit tangent vector $v$ to $e$ at $y$ lies in the tangent plane of $F$ to $y$.  Then the unit tangent vector to $\tilde e$ at $\tilde y$ lies in the tangent plane of the surface $\tilde F$. But this implies that the angle between this vector and the normal to $\tilde F$ at $\tilde y$ is $\pi/2$. Since $F$ has $(\eta, \delta/5)$-almost constant normals near $x$, the angle between this normal to $\tilde F$ and the unit normal $v_{\tilde F}$ to $\tilde F$ is at most $\delta/5$ or at least $\pi - \delta/5$. Hence, the angle between $v_{\tilde F}$ and the unit tangent vector to $\tilde e$ is at least $\pi/2 - \delta/5$ and at most $\pi/2 + \delta/5$. This contradicts the assumption that $v_{\tilde F}$ is $(\eta, 2\delta)$-nonperpendicular to the edges of $\Delta$.

A similar argument gives that $F$ intersects each face transversely. For otherwise there is some point $y$ in the intersection between $F$ and some face $R$ of $\Delta$ such that the unit normals for $F$ and $R$ coincide at that point. Pick any edge $e$ of $\Delta$ contained in that face and any point $y'$ on $e$. Then the unit normals of $\tilde R$ at the two points $\tilde y$ and $\tilde y'$ differ by at most $\delta$, by Definition \ref{SufficientlyTransverse} (4). Hence, the angle between the unit tangent vector to $e$ at $\tilde y'$ and the unit normal to $\tilde F$ at $\tilde y$ differ by at least $\pi/2-\delta$ and at most $\pi/2 + \delta$. Again, this contradicts the assumption that $v_{\tilde F}$ is $(\eta, 2\delta)$-nonperpendicular to the edges of $\Delta$.

Hence, $F \cap \Delta$ is a properly embedded surface in $\Delta$, with boundary transverse to $\partial \Delta$.

Suppose there is a component $D$ of $F \cap \Delta$ that is not a normal disc. Then at least one of the following cases must hold:
\begin{enumerate}
\item $D$ intersects an edge of $\Delta$ more than once; 
\item some component of intersection between $D$ and a face of $\Delta$ is a closed curve;
\item $D$ is not a topological disc.
\end{enumerate}

\textbf{Case 1}.  Suppose $D$ intersects an edge $e$ of $\Delta$ more than once. 
Then there is a disc $C$ embedded in $\Delta$ such that $\partial C$ is a concatenation of an arc in $e$ and
an arc $\alpha$ in $F$, and with the remainder of some collar neighbourhood of $\partial C$ in
$C$ disjoint from $\partial \Delta$ and $F$. Let $y_1$ and $y_2$ be the
two points of $F$ at the endpoint of the arc $e \cap C$. Pick a continuous choice of unit normals over $\tilde F$ i.e. all normals lie on the same side of the two-sided surface. Since $F$ is
sufficiently transverse to $\Delta$, the unit normal $v_{\tilde F}$ to $\tilde F$ 
has angle at most $\pi/2 - 2\delta$ or at least $\pi/2 + 2 \delta$ from
the tangent vectors of $\tilde e$ at $\tilde y_1$ and $\tilde y_2$. As $F$ has $(\eta, \delta/5)$-almost constant normals, the
normal vector to $\tilde F$ at $\tilde y_1$ has angle at most $\pi/2 - \delta$ 
or at least $\pi/2 + \delta$ with the tangent vector of $\tilde e$ at $\tilde y_1$. 
Say that this angle at most $\pi/2 - \delta$. The same is true at $\tilde y_2$, but because of the path $\alpha$,
the unit normals to $\tilde F$ at $\tilde y_1$ and $\tilde y_2$  either both point away from each
other or both point towards each other. Hence, the angle between the tangent vector of $\tilde e$ at
$\tilde y_2$ and the unit normal of $\tilde F$ at $\tilde y_2$ is at least $\pi/2 +\delta$.
So, the angle between $v_{\tilde F}$ and the unit tangent vector of $\tilde e$ at $\tilde y_1$ is at most
$\tilde y_1$ is at most $\pi/2$. And the angle between $v_{\tilde F}$ and the unit tangent vector of $\tilde e$ at $\tilde y_2$ is at most
$\tilde y_2$ is at least $\pi/2$. As we travel along $\tilde e$ from $\tilde y_1$ to $\tilde y_2$,
the tangent vectors to $\tilde e$ vary continuously. So, there is some point $\tilde y$ on $\tilde e$
where the tangent vector to $\tilde e$ has angle exactly $\pi/2$ with $v_{\tilde F}$. This contradicts the assumption
that $v_{\tilde F}$ is $(\eta, 2\delta)$-transverse to $\Delta$.

\textbf{Case 2}. Suppose that for some face $R$ of $\Delta$, there is a simple closed curve of $R \cap F$.

Let $C$ be a simple closed curve of $R \cap F$. This corresponds to a curve $\tilde C$ of $\tilde R \cap \tilde F$. While $\tilde R$ is close to being Euclidean, it is not strictly Euclidean. At each point of $\tilde C$, its unit tangent lies in a tangent plane to $\tilde R$. Translate each of
these planes so that it goes through the origin in $T_xM$. Each such plane intersects the unit sphere in $T_xM$
in a great circle. As one goes around $\tilde C$, these great circles may vary, but only by an angle at most $\delta$.
Hence, the unit tangent vectors of $\tilde C$ lie in an annular strip $A$ in this unit sphere of width at most $\delta$,
and as one goes around $\tilde C$, the tangent vector to $\tilde C$ traces out a closed path that goes once
around this annular strip. (See Figure \ref{Fig:SimpleCurveFace}.)

Pick an $e$ edge of $R$ and some point $y$ on $e$. This corresponds to an edge $\tilde e$ of $\tilde R$ and a point $\tilde y$ in $T_xM$.
Translate the tangent plane of $\tilde R$ at $\tilde e$ so that it goes through the origin in $T_xM$. Then the unit tangent to $\tilde e$
becomes a point $v$ on the unit sphere. It lies inside the annular strip $A$. There must be a tangent vector to $\tilde C$
within angle $\delta$ from $v$. So the normal to $\tilde F$ at this point has angle more than $\pi/2 - \delta$ and less
than $\pi/2 + \delta$ from $v$. So, the normal $v_{\tilde F}$ has angle more than $\pi/2 - 2 \delta$ and less than $\pi/2 + 2 \delta$ from $v$.
This contradicts the fact that $v_{\tilde F}$ is $(\eta, 2\delta)$-transverse to $\Delta$.

\begin{figure}
  \includegraphics[width=0.95\textwidth]{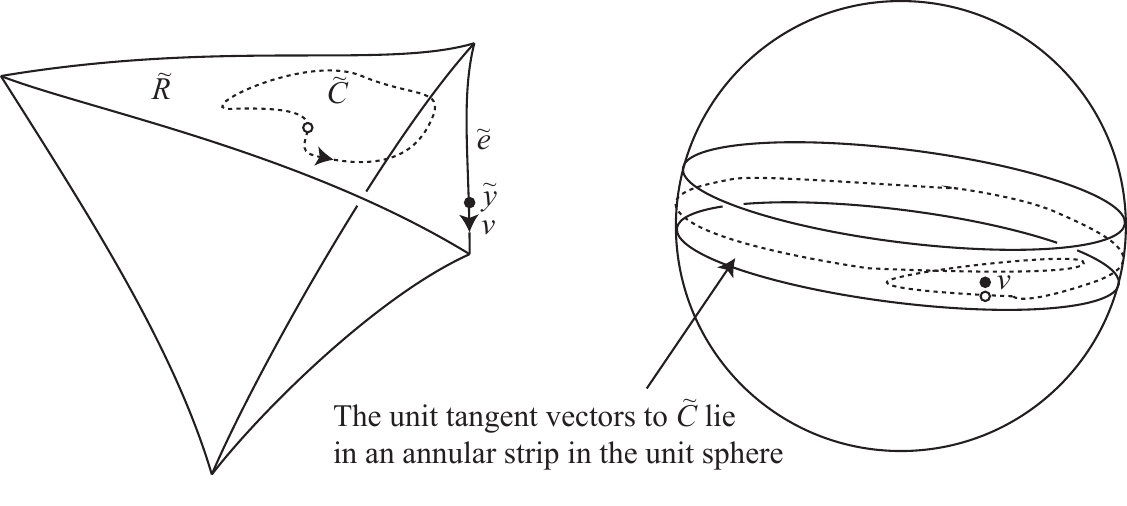}
  \caption{A simple closed curve $C$ of $R \cap F$ corresponds to a simple closed curve $\tilde C$ of $\tilde R \cap \tilde F$. Some unit tangent vector of $\tilde C$ is almost parallel to a unit tangent vector at some point $\tilde y$ on some edge of $\tilde R$.}
  \label{Fig:SimpleCurveFace}
\end{figure}

\textbf{Case 3.} Suppose that $D$ is not a disc, but that Cases 1 and 2 do not hold. 

Since $F$ has $(\eta, \delta/5)$-almost constant normals near $x$ and $\delta < \pi/6$, no unit tangent vector to $\tilde{F}$ equals the normal $v_{\tilde F}$. Consider the projection map $\mathcal{P}$ from $T_xM$ onto the plane through the origin perpendicular to $v_{\tilde F}$. Then $\mathcal{P}$ restricted to $\tilde{F} \cap B(0,\eta)$ has no critical points. It is therefore an immersion. Let $N$ be a cylinder in $T_xM$ centred at the origin (i.e. the origin is the midpoint of its height), with the two end discs perpendicular to $v_{\tilde F}$. Pick $N$ so that the two end discs have radius $\eta/2$ and so that its height is  $\eta$. Thus, $N$ lies in the ball of radius $\eta$ about the origin. Moreover, since $\Delta$ lies in the ball of radius $\eta/5$ about $x$, a homeomorphic copy $\tilde \Delta$ lies in $N$. See Figure \ref{Fig:TetInCylinder}. Let $\tilde S$ be the component of $\tilde F \cap N$ containing $\tilde D$. 

\begin{figure}
  \includegraphics[width=0.4\textwidth]{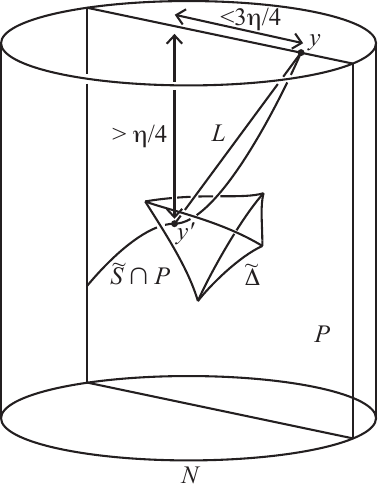}
  \caption{The cylinder $N$ in $T_xM$ containing the tetrahedron $\tilde \Delta$.}
  \label{Fig:TetInCylinder}
\end{figure}

\textbf{Claim 1.} The boundary of $\tilde S$ lies in the curved part of $\partial N$. 

If Claim 1 does not hold, the boundary of $\tilde S$ intersects one of the two end discs. Pick a point $y$ in the intersection between $\tilde S$ and the end disc, and pick a point $y'$ in $\tilde D$. Join these by a Euclidean straight line $L$. Since $y'\in\tilde{D}\subset\tilde{\Delta}$, the point $y'$ has distance at most $\eta/4$ from the origin. Hence, the distance from $y'$ to the end disc containing $y$ is at least $\eta/4$. The distance in the plane of the end disc, which is of radius $\eta/2$, 
between $y$ and the vertical projection of $y'$ is then at most $\eta/4+\eta/2=3\eta/4$.  Hence, from the right-angled triangle with sides $\eta/4, 3\eta/4$, the angle that $L$ makes from the plane containing the end disc is at least $\tan^{-1}(1/3)$. 

Let $P$ be the plane through $L$ perpendicular to the end discs. We can view $\tilde S \cap P$ as the graph of a function. By the mean value theorem applied to this function, there is a tangent vector to $\tilde S$ that is parallel to $L$. Since $\tan^{-1}(1/3) > \pi/10$, we deduce that at this point of $\tilde S$, its normal vector
differs from $v_{\tilde F}$ by more than $\pi/10$. But this contradicts the assumption that $F$ has $(\eta, \delta/5)$-almost constant normals near $x$, since $\delta<\pi/2$. This proves Claim 1. 

By Claim 1, the restriction of $\mathcal{P}$ to $\tilde S$ is an immersion onto a disc. It is therefore a covering map of the disc, and since the disc is simply-connected, it is a homeomorphism. We deduce that $\tilde S$ is a disc. Hence, $D$ is planar (i.e. has genus zero), since $\tilde D$ is a subset of $\tilde S$. But its boundary is connected, because it is a concatenation of at most four normal arcs in $\partial \Delta$, since Cases 1 and 2 do not hold. Hence, $D$ is a disc. Thus, Case 3 does not hold.

Thus, we have shown that Cases 1, 2, and 3 do not hold, and hence $F$ intersects $\Delta$ in normal discs. 

We now prove the final claim of the proposition. Consider a point $p$ in $F$, and suppose that the ball of radius $\eta/2$ about $p$ intersects two normal discs $D_1$ and $D_2$ of $\Delta \cap F$. These two discs are joined by a path of length at most $\eta$ in $F$. Furthermore, $D_1$ and $D_2$ both have distance at most $\eta/5$ from $x$ by Definition \ref{SufficientlyTransverse}(1). So, $D_1$ and $D_2$ lie in the same component $F'$ of $F \cap B(x,\eta)$. Now, for any two normal discs in $\Delta$, there must be some edge that they both intersect. Say that $D_1$ and $D_2$ both intersect the edge $e$. Pick a transverse orientation on $F'$, and pick an orientation on $e$. Using these orientations, we can make sense of the assertion that at any point $z$ of $F' \cap e$, the transverse orientation of $F'$ agrees or disagrees with the orientation on $e$, as follows. The unit normal to $\tilde F'$ at $\tilde z$ has angle at most $\delta/5$ from $v_{\tilde F}$ or $-v_{\tilde F}$, by (2) in the definition of sufficiently transverse. Hence, by (3) in that definition, the unit normal to $\tilde F'$ and the tangent to $\tilde e$ at $\tilde z$ have angle at most $\pi/2 - 9\delta/5$ or at least $\pi/2 +9 \delta/5$. In the former case, we say that the orientations of $F'$ and $e$ are consistent at $z$, and in the latter case, inconsistent. Now, at two points of $F' \cap e$ that are adjacent in $e$, the orientations at these points cannot be both consistent or both inconsistent, since we could then construct a closed loop in $B(x,\eta) $ that intersects $F'$ once.  The loop would then have to be homologically non-trivial in $B(x,\eta)$, which is a 3-ball, and hence this is a contradiction. Now fix $z_1$ to be one of the points of $F' \cap e$ and let $z_2$ be an adjacent point. Say that the orientations are consistent at $z_1$ and inconsistent at $z_2$. The tangents to $\tilde e$ at $\tilde z_1$ and $\tilde z_2$ differ by an angle at most $\delta$, by (4) of Definition \ref{SufficientlyTransverse}. Hence, the normal to $\tilde F'$ at $z_2$ has angle more than $\pi/2 + 4\delta/5$ with the tangent to $\tilde e$ at $\tilde z_1$. Now consider a path in $F'$ joining $z_2$ to $z_1$. At some point of this path, the normal to $\tilde F'$ has angle $\pi/2$ with the tangent to $\tilde e$ at $\tilde z_1$. Hence, by (2) of Definition \ref{SufficientlyTransverse}, the normal $v_{\tilde F}$ and the tangent to $\tilde e$ at $\tilde z_1$ have angle between $\pi/2 - \delta/5$ and $\pi/2 + \delta/5$. This contradicts the assumption (3) of Definition \ref{SufficientlyTransverse} that $v_{\tilde F}$ is $(\eta, 2\delta)$-nonperpendicular to the edges of $\Delta$.
\end{proof}

\section{Barycentric subdivision}
\label{Sec:Barycentric}

Our goal is find a triangulation of a codimension zero submanifold of the hyperbolic 3-manifold $M$ that can be made sufficiently transverse to any stable minimal surface. We will start with the triangulation $T$ from Theorem \ref{Thm:ExtensionBreslin} that is $(\theta, a\mu, b \mu)$-thick, for some suitably chosen $\mu > 0$. Recall that $\theta$, $a$ and $b$ are fixed constants provided by Theorem \ref{Thm:ExtensionBreslin}. Our required triangulation will, combinatorially, be equal to the barycentric subdivision $T'$ of $T$. In this section, we will define a geometric position for the vertices of this triangulation, and then we will perturb the triangulation, by moving its vertices, so that it becomes sufficiently transverse to a given surface $F$.

Consider a triangle in the hyperbolic plane, with geodesic edges. Pick an ordering of its vertices $v_1, v_2, v_3$. Define the \emph{ordered barycentre} of the triangle as follows. First pick the point half way along the geodesic joining $v_1$ and $v_2$. Then run a geodesic from this point to $v_3$. Stop one third of the way along this geodesic.

Let $\Delta$ be a hyperbolic tetrahedron. Pick an ordering $v_1, v_2, v_3, v_4$ of its vertices. Define the \emph{ordered barycentre} of $\Delta$ as follows. Starting at the ordered barycentre of $v_1, v_2, v_3$, run a geodesic to $v_4$. The point that is one quarter of the way along this geodesic is the ordered barycentre of $\Delta$.

The \emph{ordered barycentric subdivision} of $\Delta$ is the geodesic triangulation, where each vertex is the ordered barycentre of a simplex of $\Delta$. Since vertices are simplices as well, we include each old vertex together with the new ones. Here, the ordering on the vertices of each simplex is the restriction of the given ordering on the vertices of $\Delta$. 

Now let $T$ be a geodesic triangulation of a 3-dimensional submanifold of a hyperbolic 3-manifold. Pick a total ordering on its vertices. Then the \emph{ordered barycentric subdivision} of $T$ is obtained by replacing each tetrahedron of $T$ with its ordered barycentric subdivision. 

\begin{lemma}
\label{Lem:Barycentric}
Given $\theta \in (0,\pi)$ and $0 < a \leq b \leq 1/40$, there are $\theta' \in (0,\pi)$ and $0 < a' \leq b' \leq 1/40$
 with the following property. If $\mu$ is any positive real number at most the 3-dimensional Margulis constant, and $\Delta$ is a hyperbolic tetrahedron that is $(\theta, a \mu, b \mu)$-thick, then each tetrahedron of the ordered barycentric subdivision of $\Delta$ is $(\theta',a' \mu,b' \mu)$-thick.

\end{lemma}

\begin{proof}
We will consider the set of all $(\theta, a \mu, b \mu)$-thick hyperbolic tetrahedra in hyperbolic 3-space $\mathbb{H}^3$, with one specific vertex at a fixed basepoint in  $\mathbb{H}^3$. The remaining three vertices are points of $\mathbb{H}^3$, and so we can view each such hyperbolic tetrahedron as a point in $\mathbb{H}^3 \times  \mathbb{H}^3 \times  \mathbb{H}^3$. We have an upper bound on $\mu$. \emph{Suppose} that we also had a positive lower bound of the form $\mu \geq c$ for some constant $c >0$. Then the set of all $(\theta, a \mu, b \mu)$-thick tetrahedra would form a compact subset of $\mathbb{H}^3 \times  \mathbb{H}^3 \times  \mathbb{H}^3$. Moreover, the minimum interior angle between adjacent faces of the ordered barycentric subdivision is a continuous function on this compact set. Hence, its minimum is realised by some $(\theta, a \mu, b \mu)$-thick tetrahedron. We would therefore obtain a positive lower bound on these angles. Similarly, we would obtain positive upper and lower bounds on the edge lengths of the tetrahedra in the ordered barycentric subdivision. However, we are not assuming a positive lower bound on $\mu$. We therefore argue as follows.

We let $\mathbb{H}^3_{1/\mu}$ denote the Poincar\'e ball model of hyperbolic space, rescaled so that it has Euclidean radius $1/\mu$. We place $v_1$ at the origin of this ball.
Then each choice of $v_2, v_3, v_4$ determines a point in $\mathbb{H}^3_{1/\mu} \times \mathbb{H}^3_{1/\mu} \times \mathbb{H}^3_{1/\mu}$. A point $v$ at hyperbolic distance $d$ from $v_1$ has Euclidean distance $(1/\mu) \tanh(d/2)$ from the origin. Hence, if the 4 points determine a tetrahedron that is $(\theta, a \mu, b \mu)$-thick, then $v_2, v_3, v_4$ lie in the shell between radius $(1/\mu) \tanh(a\mu/2)$ and $(1/\mu) \tanh(b\mu/2)$ about the origin. Recall that $\mu$ is between $\epsilon_3$ and 0, and as $\mu$ approaches 0, the tetrahedra increasingly approximate Euclidean ones. This yields an upper and lower bound on the side lengths of the tetrahedra as follows. Let $S$ denote the compact subset of $\mathbb{E}^3$ consisting of points with distance from the origin between $(1/\epsilon_3) \tanh(a\epsilon_3/2)$ and $b/2$. We encode the tetrahedron using the three points in $S$, as well as the parameter $\mu$. Thus, the tetrahedron corresponds to a point in $S \times S \times S \times (0,\epsilon_3]$. Let $C$ denote the set of points in $S \times S \times S \times (0,\epsilon_3]$ corresponding to $(\theta, a \mu, b \mu)$-thick tetrahedra. Let $\overline{C}$ be its closure in $S \times S \times S \times [0,\epsilon_3]$. Each point in $\overline{C} \setminus C$ determines a Euclidean tetrahedron that is $(\theta, a, b)$-thick, since, again, as $\mu \rightarrow 0$, the tetrahedra in $S \times S \times S \times (0,\epsilon_3]$ increasingly approximate Euclidean ones. Hence, for each point in $\overline{C}$, there is a well-defined ordered barycentric subdivision lying either in $\mathbb{H}^3$ or $\mathbb{E}^3$. In this subdivision, the minimal interior angle of the tetrahedra is a continuous function on $\overline{C}$. Hence, by compactness of $\overline{C}$, this minimal interior angle is realised by taking some $(\theta, a\mu, b\mu)$-thick hyperbolic tetrahedron $\Delta$ and barycentrically subdividing, or by taking some $(\theta, a, b)$-thick Euclidean tetrahedron and barycentrically subdividing. In particular, there is a positive lower bound $\theta'$ for these interior angles. Similarly, we can also consider the function on $C$ that is the ratio of the minimal edge length of the tetrahedra in the barycentric subdivision to $\mu$. This is a continuous function that extends to a continuous function on $\overline{C}$. Hence, again, it has a positive lower bound $a'$. Therefore, the edge length of each tetrahedron in the barycentric subdivision of $\Delta$ is at least $a' \mu$. By a similar argument, there is also an upper bound $b'\mu$ to these edge lengths. Hence, given any $(\theta, a\mu, b\mu)$-thick tetrahedron $\Delta$, the tetrahedra of the ordered barycentric subdivision of $\Delta$ are $(\theta',a'\mu,b'\mu)$-thick. \end{proof}

\begin{remark}\label{cc}
The argument in the above proof can be applied more generally, as a method of establishing properties of $(\theta, a \mu, b\mu)$-thick tetrahedra, where $\mu \in (0,\epsilon_3]$. Let $C$ be the space of such tetrahedra, viewed as a subset of $S \times S \times S \times (0,\epsilon_3]$, and let $\overline{C}$ be its closure in $S \times S \times S \times [0,\epsilon_3]$. Suppose that $f \colon C \rightarrow \mathbb{R}$ is some continuous function that extends continuously to $\overline{C}$. For example, in the above proof, we took $f$ to be the minimum dihedral angle of the tetrahedra in the barycentric subdivision of the given $(\theta, a \mu, b\mu)$-thick tetrahedron. The continuity of $f$ and the compactness of $\overline{C}$ implies that $f$ is bounded below and that its minimal value is achieved at some point of $\overline{C}$. If we can interpret the restriction of $f$ to $\overline{C} - C$ as a geometric quantity associated with Euclidean tetrahedra, then we can hope to show that this minimal value is strictly positive. We will call this style of argument as `using continuity and compactness'. 
\end{remark}

Remark \ref{cc} implies the following lemma. Here, we consider the faces of a hyperbolic tetrahedron and the three interior angles in each face. Setting $f$ to be the minimal such angle over all the faces, we obtain this result.

\begin{lemma}
\label{Lem:InteriorAngleFace}
Given $\theta \in (0,\pi)$ and $0 < a \leq b \leq 1/40$,
there is a positive $\theta_2$ with the following property. If $\mu$ is any positive real number at most the 3-dimensional Margulis constant, and $\Delta$ is a hyperbolic tetrahedron that is $(\theta, a \mu, b \mu)$-thick, then for each of its faces, each interior angle in that face is at least $\theta_2$. 
\end{lemma}

We are now ready to adjust the barycentric subdivision.

\begin{lemma}
\label{Lem:AvoidVectorsForBarycentre}
For any $0 < a \leq b \leq 1/40$ and $0 < \theta < \pi$, there are $\delta'' > 0$ and $\theta'' \in (0, \pi)$ and $0 < a'' \leq b'' \leq 1/40$ with the following properties. For any $0 < \mu \leq \epsilon_3$, any $(\theta, a\mu, b\mu)$-thick hyperbolic tetrahedron $\Delta$, any total ordering on its vertices, and any assignment of unit vectors $v_1, \dots, v_4$ at the vertices of $\Delta$, there is a perturbation of the ordered barycentre of $\Delta$ and the ordered barycentres of its 2-dimensional faces, so that the following hold:
\begin{enumerate}
\item the perturbed barycentre of each face still lies in that face;
\item the resulting triangulation of $\Delta$ is $(\theta'', a''\mu, b''\mu)$-thick;
\item for each edge $e$ of the perturbed barycentric subdivision of $\Delta$, other than an edge lying in an edge of $\Delta$, and each vertex $x$ of $\Delta$,
some unit tangent vector to $\exp_x^{-1}(e)$ has angle at most $(\pi/2) - \delta''$ or at least $(\pi/2) + \delta''$ from the given unit vector $v_i$ at $x$.
\end{enumerate}
\end{lemma}

\begin{proof}
By Lemma \ref{Lem:Barycentric}, there are constants $0 < a' \leq b' \leq 1/40$ and $\theta' > 0$ such that the barycentric subdivision of $\Delta$ is $(\theta', a'\mu, b'\mu)$-thick. Hence, by Lemma \ref{Lem:InteriorAngleFace}, there is a constant $c > 0$ such that the distance of the barycentre of each 2-dimensional face of $\Delta$ from each edge of $\Delta$ is at least $c \mu$. We will perturb this barycentre by a distance at most $c \mu / 2$. Hence, it will remain in the interior of the face. The possible location of the perturbed point lies in a disc $D$ of radius $c \mu/2$ about the barycentre, and this disc has area at least $\pi (c \mu /2)^2$. Running from each of the three vertices of the face to the perturbed barycentre are three new edges. Similarly running from the midpoints of the edges of $\Delta$ are three new edges that run to the perturbed barycentre. We need to ensure that their unit tangent vectors at these six points $p_1, \dots, p_6$ have angle at most $(\pi/2) - \delta''$ or at least $(\pi/2) + \delta''$ from each of the given vectors $v_1, \dots, v_4$. Thus, emanating from $p_i$, there is a region of forbidden points where the perturbed barycentre may not lie. The intersection of this forbidden region with $D$ has area at most $d \mu^2 \delta''$ for some constant $d$. The reason is as follows. Each vector $v_j$ and each $p_i$ give rise to a region in $D$ that is enclosed by a quadrilateral, with two of its sides being the intersections of $D$ with two rays emanating from the relevant $p_i$ (see Figure \ref{Fig:ForbiddenRegion}), and two other sides being tangent to $D$ and transverse to the rays. Hence two of those edges have length at most a diameter of $D$, and the other two edges have length at most a constant times $\mu \delta''$. (This can be verified with a quick argument involving the law of sines. Here, we use the assumption that $\mu \leq \epsilon_3$. Any other universal upper bound on $\mu$ would have sufficed here.) Thus, the six forbidden regions intersect $D$ in a set of area at most $6d \mu^2 \delta''$. This is less than  $\pi (c \mu /2)^2$ provided $\delta''$ is small enough. Hence, provided $\delta''$ is sufficiently small, we may perturb the barycentres of the faces by distance at most $c \mu / 2$, and thereby ensure that condition (3) holds for the edges lying in the 2-dimensional faces of $\Delta$.

\begin{figure}
  \includegraphics[width=0.6\textwidth]{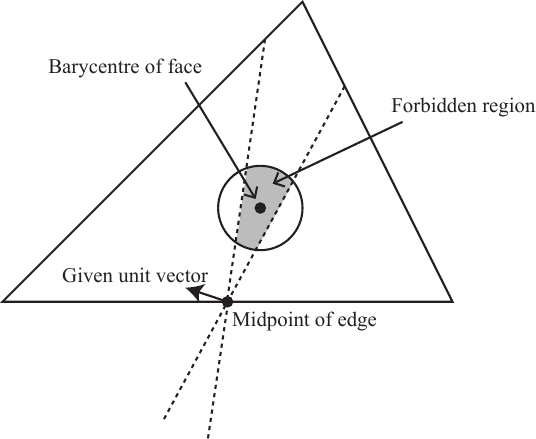}
  \caption{A unit vector has been specified. This imposes a region in which the perturbed barycentre must not lie. The disc $D$ shows the points within $c\mu/2$ from the barycentre. Shaded within this disc is the forbidden region that the perturbed barycentre must avoid.}
  \label{Fig:ForbiddenRegion}
\end{figure}

A similar argument works for the barycentre of $\Delta$. There is a constant $c' > 0$ such that the distance of the barycentre is at least $c' \mu$ from each of the faces of $\Delta$. We will perturb the barycentre by a distance at most $c' \mu / 2$. Thus, the volume of the ball $B$ of possible perturbations is at least $(4\pi/3) (c' \mu / 2)^3$. Again, there are forbidden regions where the perturbed barycentre is not allowed to lie. These arise from the 14 edges of the ordered barycentric subdivision ending at the barycentre of $\Delta$ and from the four given vectors $v_1, \dots, v_4$. The intersection of each forbidden region with $B$ has volume at most a constant times $\mu^3 \delta''$. (This is because the region lies in a ball of radius $b' \mu$ centred at a vertex of the perturbed barycentric subdivision of $\Delta$. Consider an annular strip in the boundary of this ball with the core curve being a great circle and with angular width $2 \delta''$. Now form the cone over this annular strip with cone point being the centre of the ball. This intersection between each forbidden region and $B$ lies within such a cone.) Hence, if $\delta''$ is sufficiently small, then there is a perturbation that avoids all of these forbidden regions.

Finally, we verify the resulting triangulation of $\Delta$ is $(\theta'', a''\mu, b''\mu)$-thick for some constants  $\theta'' \in (0,\pi)$ and $0 < a'' \leq b'' \leq 1/40$. As in Lemma \ref{Lem:Barycentric}, this follows from compactness. The set of possible perturbations is compact. Hence, the minimum interior angle is realised by some $\Delta$ and some perturbation of its barycentre and the barycentres of its faces. In particular, this minimum is bounded below by some $\theta'' > 0$. Similarly, the ratios of the minimum and maximum edge lengths to $\mu$ are also realised by some specific examples, and so there are the constants $a''$ and $b''$. \end{proof}

The above lemma is used to perturb the barycentre of each tetrahedron and the barycentre of each of its faces. But it leaves the midpoint of each edge of the original triangulation untouched. Neighbourhoods of such edges are shown on the right of Figure \ref{Fig:BiCone}. (Such an edge is the edge $e$ in the figure.) We now analyse the geometry of these neighbourhoods.

A \emph{triangulated bi-cone} is a convex hyperbolic polyhedron $P$ with a geodesic triangulation $T$ constructed as follows. Pick a  triangulation $T_1$ of the circle or interval. Let $T_2$ be the triangulation of the interval with a single 1-simplex. Then $T$ is the join of $T_1$ and $T_2$.

Thus, a triangulated bi-cone has the following properties, which could be taken as an alternative definition. There is a specified edge $e$ of $T$, called its \emph{central}
edge. Its endpoints lie on $\partial P$, but its interior may lie in $\mathrm{int}(P)$ or in $\partial P$. For every tetrahedron $\Delta$ in $T$, the following conditions hold:
\begin{enumerate}
\item $e$ is an edge of $\Delta$; 
\item the two faces of $\Delta$ that do not contain $e$ lie on the boundary of $P$;
\item if interior of $e$ lies in the interior of $P$, then the remaining two faces of $\Delta$ have interior in the interior of $P$ (Figure \ref{Fig:BiCone}, the top left polyhedron); otherwise the same holds but with the exception of exactly two faces of $P$ incident to $e$ (Figure \ref{Fig:BiCone}, the bottom left polyhedron).
\end{enumerate} 
In the description of $T$ as the join of the triangulations $T_1$ and $T_2$, this edge $e$ is the copy of $T_2$ in the join.

The \emph{canonical subdivision} of a triangulated bi-cone has one extra vertex, at the midpoint of $e$, and each of the tetrahedra is subdivided into two. (See Figure \ref{Fig:BiCone}: polyhedra on the right are shown with their canonical subdivision.)
This new vertex is called the \emph{new vertex}. This subdivision can also be described by subdividing $T_2$ into two 1-simplices, giving a new triangulation $T'_2$ of the interval. Then the canonical subdivision is the join of $T_1$ and $T'_2$.

\begin{figure}
  \includegraphics[width=0.8\textwidth]{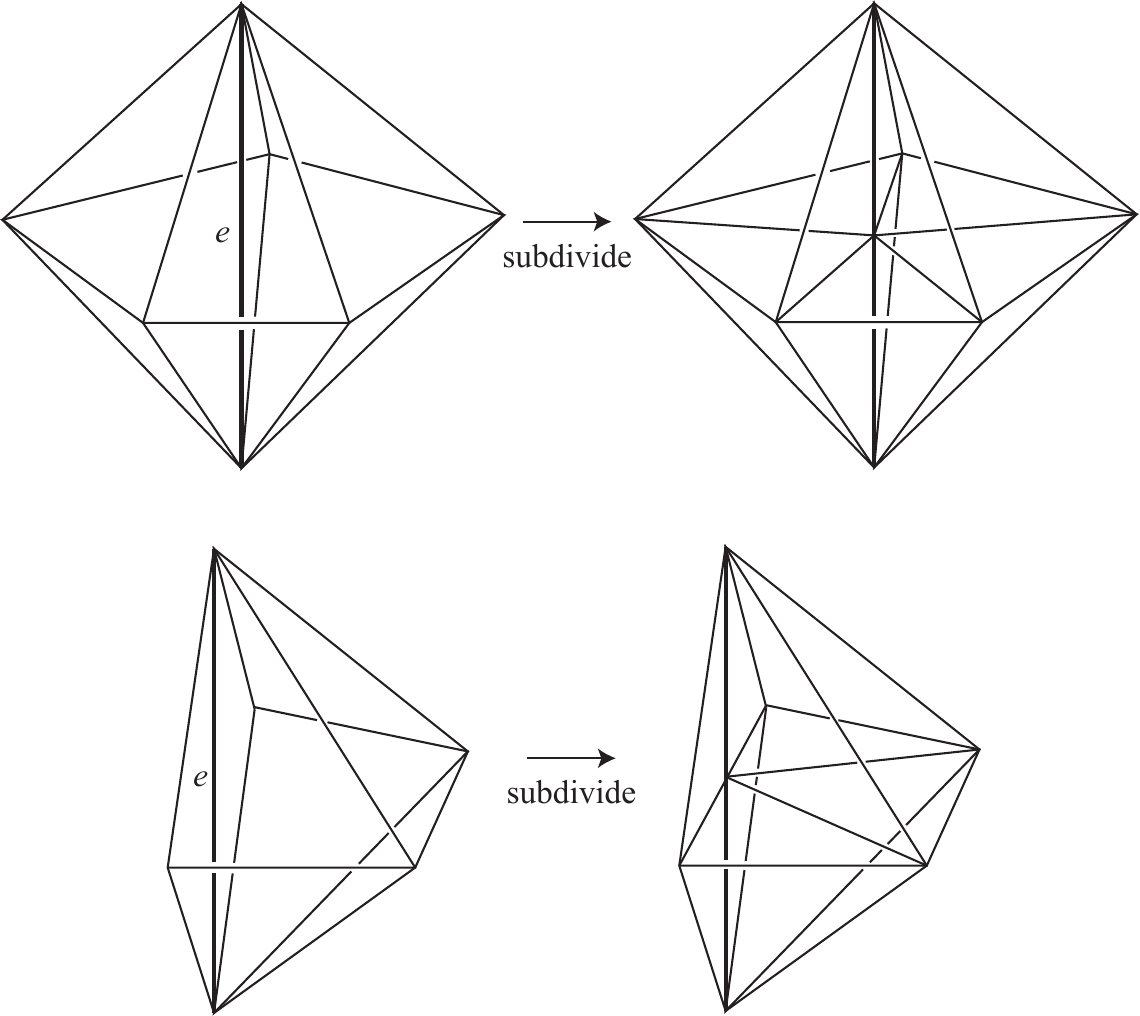}
  \caption{Left: triangulated bi-cones with their central edge $e$ shown in bold. Right: their canonical subdivisions.}
  \label{Fig:BiCone}
\end{figure}

\begin{lemma} 
\label{Lem:BiCone}
For any $0 < a'' \leq b'' \leq 1/40$ and $\theta'' \in (0,\pi)$, there are $\delta''' \in (0,\pi/2)$ and $\theta''' \in (0,\pi)$ and $0 < a''' \leq b''' \leq 1/40$
 with the following properties. Let $T$ be the canonical subdivision of a triangulated bi-cone $P$. Suppose that $T$ is $(\theta'', a''\mu, b''\mu)$-thick for some positive $\mu \leq \epsilon_3$. Suppose that at both vertices at the endpoints of its central edge, a unit vector has been chosen. Then there is a perturbation of the new vertex so that the resulting triangulation is $(\theta''', a'''\mu, b'''\mu)$-thick and so that the unit vectors in the direction of the edges incident to the new vertex have angle at most $(\pi/2) - \delta'''$ or at least $(\pi/2) + \delta'''$ from the given vectors. In the case where the new vertex lies on the boundary of $P$, then it is perturbed into the interior of $P$.
\end{lemma}

\begin{proof}
This is very similar to the argument in Lemma \ref{Lem:AvoidVectorsForBarycentre}. There is a constant $c''$ such that the new vertex of $T$ has distance at least $c'' \mu$ from the faces in  $\partial P$ that do not include the new vertex. We consider the intersection between $P$ and the ball of radius $(c'' \mu/2)$ about the new vertex as the set $B$ of possible perturbations. There are forbidden regions coming from the two given vectors. These intersect $B$ in regions with volume at most a constant times $\mu^3 \delta'''$. Hence, provided 
$\delta'''$ is sufficiently small, a perturbation can be found. 

The resulting triangulation is $(\theta''', a'''\mu, b'''\mu)$-thick, again by a continuity and compactness argument (Remark \ref{cc}). Indeed, note that there is a uniform upper bound to the number of tetrahedra around the central edge of $P$. This is because the interior angles of each tetrahedron incident to the edge are at least $\theta''$, and hence the number of such tetrahedra is at most $2\pi/\theta''$ before subdivision and at most $4 \pi / \theta''$ after subdivision. Hence, the set of combinatorial possibilities for $P$ is finite, and for each such possibility, the space of possible $(\theta'', a''\mu, b''\mu)$-thick triangulations is given as a subset $C$ of $\mathbb{H}^3_{1/\mu} \times \dots \times \mathbb{H}^3_{1/\mu} \times (0,\epsilon_3]$. As in the proof of Lemma \ref{Lem:Barycentric}, we may take the closure $\overline{C}$ by adding the limit points as $\mu \rightarrow 0$. This is compact, and so again by continuity, there are suitable $\theta'''$ and $0 < a''' \leq b''' \leq 1/40$, as required. \end{proof}

We now fix $\theta$, $a$ and $b$ for a triangulation from Theorem \ref{Thm:ExtensionBreslin}. Lemma \ref{Lem:Barycentric} then provides $\theta', a', b'$ for its barycentric subdivision, and from this we obtain $\theta'', a'', b''$ for the perturbed barycentric subdivision from Lemma \ref{Lem:AvoidVectorsForBarycentre} (perturbing the barycenters of 2-simplices and 3-simplices) and hence $\theta''', a''', b'''$ from Lemma \ref{Lem:BiCone} (perturbing the barycenters of 1-simplices). Let $\theta'''' = (4/5)\theta'''$, let $a'''' = (4/5) a'''$ and let $b'''' = (5/4)b'''$.

\begin{proposition}\label{mu}
There is a constant $\mu > 0$ with the following property. Let $M$ be a hyperbolic 3-manifold and let $T$ be a $(\theta, a\mu, b\mu)$-thick triangulation of a 3-dimensional submanifold of $M_{[\mu/4, \infty)}$. Pick a total ordering on its vertices. Let $F$ be any stable minimal surface with trivial normal bundle properly embedded in $M$.
 Then there is a perturbation of the ordered barycentric subdivision of $T$ that is $(\theta'''', a''''\mu, b''''\mu)$-thick and that is sufficiently transverse to $F$ at scale $\mu/8$.
 \end{proposition}

\begin{proof} 
Lemmas \ref{Lem:AvoidVectorsForBarycentre} and \ref{Lem:BiCone} provide constants $\delta'' > 0$ and $\delta'''> 0$. Set 
$\delta = \min \{ \delta'', \delta''' \} /4$.

Then Corollary \ref{Cor:AlmostConstantNormalsSimplified} (applied using $\delta/30$) provides $\eta > 0$ such that near each point of $M$, $F$ has $(\eta \delta/(15\pi), \delta/5)$-almost constant normals. If necessary, we reduce $\eta$ so that the following properties also hold. For any totally geodesic plane $\Pi$ in $\mathbb{H}^3$ and any point $p$ in $\mathbb{H}^3$, the angle between the unit normals at any two points of $\exp_p^{-1}(\Pi) \cap B(0, \eta)$ differ by at most $\delta$, by Lemma \ref{Cor:AlmostConstantNormalPlane}.
Similarly, for any geodesic $e$ in $\mathbb{H}^3$ and any point $p$ in $\mathbb{H}^3$, the angle between the unit tangent vectors at any two points of $\exp_p^{-1}(e) \cap B(0, \eta)$ differ by at most $\delta$.

We set $\mu = 8\eta \delta / (15\pi)$. Then $F$ has $(\mu/8, \delta/5)$-almost 
constant normals near every point of $M$. 
Note that $\delta \leq \delta''' \leq \pi/2$. So, $\mu/8 \leq \eta/30$.

Now that we have specified $\mu$, let $T$ be a $(\theta, a\mu, b\mu)$-thick triangulation of a 3-dimensional submanifold of $M_{[\mu/4, \infty)}$.
Pick a total ordering on its vertices, and let $T'$ be the resulting ordered barycentric subdivision. We now assign a unit vector at each vertex of $T$, as follows. Consider one such point $p$. If $F$ intersects $B_M(p, \mu/8)$, pick a unit normal to some point of the surface $\tilde F = \exp_p^{-1}(F) \cap B(0,\mu/8)$. On the other hand, if $F$ is disjoint from $B_M(p, \mu/8)$, then pick a unit vector at $p$  arbitrarily. 

We now perturb the barycentre of each face of $T$ and the barycentre of each tetrahedron of $T$, using Lemma \ref{Lem:AvoidVectorsForBarycentre}. The resulting triangulation is $(\theta'', a''\mu, b''\mu)$-thick. Around each edge of $T$, the tetrahedra of $T'$ combine to form the canonical subdivision of a bi-cone. Perturb the midpoint of the edge, using Lemma \ref{Lem:BiCone}, to give a $(\theta''', a'''\mu, b'''\mu)$-thick triangulation. If $F$ intersects any vertices of this triangulation, we make an extremely small perturbation to these vertices, ensuring that the resulting triangulation is $(\theta'''', a''''\mu, b''''\mu)$-thick.
Combinatorially, this is still $T'$. 

We claim that $F$ is sufficiently transverse at scale $\mu/8$ to each tetrahedron of $T'$, as required. 
Consider any tetrahedron $\Delta'$ of $T'$. One of its four vertices is a vertex of $\Delta$. Set $x$ to be this point.

We arranged above, by our choice of $\eta$, that 
Properties (4) and (5) of Definition \ref{SufficientlyTransverse} hold.
Since $b'''' \leq 1/32$, the maximal edge length of $T'$ is at most $b''''\mu \leq \mu/32$, which gives (1) in the definition of being sufficiently transverse.
Property (2) is that $F$ has $(\mu/8, \delta/5)$-almost constant normals, which we verified above.

We now check property (3) in the definition of sufficiently transverse. If $F$ is disjoint from $\Delta'$, there is nothing to prove. So we may suppose that $\tilde F \cap \tilde \Delta'$ is
non-empty. Let $v_{\tilde F}$ be a unit normal at some point of $\tilde F \cap \tilde \Delta'$. We will show that 
$v_{\tilde F}$ is $(\eta, 2\delta)$-nonperpendicular to the edges of $\Delta'$. 

We have already picked a unit vector $v$ at $x$, which is the unit normal vector at some
point of $\tilde F$. Since $F$ has $(\mu/8, \delta/5)$-almost constant normals, we deduce that the angle
between $v$ and $v_{\tilde F}$ is at most $\delta/5$ or at least $\pi- \delta/5$.

Consider any edge
$e$ of $\Delta'$. If $e$ has at least one endpoint in a barycentre of a face of $\Delta$ or at the barycentre of $\Delta$,
then the perturbation of these barycentres in Lemma \ref{Lem:AvoidVectorsForBarycentre} ensures that
a unit tangent vector of $\tilde e$ has angle at most $(\pi/2) - \delta''$ or at least $(\pi/2) + \delta''$ from $v$.
On the other hand, if $e$ does not have any endpoint that is one of these barycentres, then
the perturbation of the edge midpoints of $T$ given in Lemma \ref{Lem:BiCone} guarantees that 
a unit tangent vector of $\tilde e$ has angle at most $(\pi/2) - \delta'''$ or at least $(\pi/2) + \delta'''$ from $v$.
Thus, the angle between $v$ and \emph{some} unit tangent vector to $\tilde e$ is at most $(\pi/2) - 4 \delta$ or at least
$(\pi/2) + 4 \delta$. So, the angle between $v$ and \emph{every} unit tangent vector to $\tilde e$ is at most $(\pi/2) - 3 \delta$ or at least
$(\pi/2) + 3 \delta$ by property (5). Hence, the angle between $v_{\tilde F}$ and every unit tangent vector to $\tilde e$ is at most $(\pi/2) - 3 \delta + (\delta/5)$ or at least $(\pi/2) + 3 \delta - (\delta/5)$.
This verifies property (3). \end{proof}

We also record some further properties of thick triangulations.

\begin{lemma}
For positive real numbers $\theta''''$, $a''''$ and $b''''$, there is a constant $k_0$ such that the number of tetrahedra in a $(\theta'''', a''''\mu, b''''\mu)$-thick triangulation incident to any vertex is at most $k_0$.
\end{lemma}

\begin{proof} 
This just follows from the fact that there is a universal lower bound on the solid angle at any vertex of a $(\theta'''', a''''\mu, b''''\mu)$-thick tetrahedron. \end{proof}

\begin{lemma}
\label{Lem:VolumeTet}
There is a constant $v$, depending on $\theta''''$, $a''''$ and $b''''$ with the following property.
Let $\Delta$ be a hyperbolic tetrahedron that is $(\theta'''', a''''\mu, b''''\mu)$-thick.
Then the volume of $\Delta$ is at least $v \mu^3$.
\end{lemma}

\begin{proof} 
This again follows from continuity and compactness (Remark \ref{cc}), by considering the function $\mathrm{volume}(\Delta)/\mu^3$ on the space of $(\theta'''', a''''\mu, b''''\mu)$-thick tetrahedra $\Delta$. \end{proof}

\begin{lemma}
\label{Lem:TetWithinBall}
There is a constant $k$, depending on $\theta''''$, $a''''$ and $b''''$ with the following property.
Let $T$ be a triangulation of a 3-dimensional submanifold of $M$ that is $(\theta'''', a''''\mu, b''''\mu)$-thick, for some
$\mu$ at most the 3-dimensional Margulis constant. Then
the number of tetrahedra of $T$ that intersect a ball of radius $b''''\mu$ about any point of $M$ is at most $k$.
\end{lemma}

\begin{proof}
Any $(\theta'''', a''''\mu, b''''\mu)$-thick tetrahedron intersecting a ball of radius $b''''\mu$ about a point $p$ lies within the ball $B$ of radius $2b''''\mu$ about $p$.
Since $\mu$ is bounded above, $B$ has volume at most $c \mu^3$ for some constant $c > 0$. By Lemma \ref{Lem:VolumeTet}, the volume of each tetrahedron is at least $v \mu^3$.
Hence, the number of tetrahedra is at most $c/v$. \end{proof}

\section{Conclusion of the proof}
\label{Sec:ConclusionOfProof}

Recall that we took a triangulation from Theorem \ref{Thm:ExtensionBreslin} and modified it in Section \ref{Sec:Barycentric} so that we now have fixed constants $\theta''''$, $a''''$, $b''''$ and $\mu$. 

\begin{proposition}\label{DiskNumber}
There is a constant $N$, depending on $\theta''''$, $a''''$, $b''''$ and $\mu$, with the following property. Let $M$ be a hyperbolic 3-manifold and
let $T'$ be a $(\theta'''', a''''\mu, b''''\mu)$-thick triangulation of a 3-dimensional submanifold of $M$ contained in $M_{[\mu/4, \infty)}$.
Suppose that $T'$ is sufficiently transverse to a properly embedded 2-sided $\pi_1$-injective minimal surface $F$ at scale $\mu/8$.
Then the number of normal discs of $F \cap T'$ is at most $N |\chi(F)|$.
\end{proposition}

\begin{proof} We can cover $F$ by discs of radius $\mu/16$,  where the number of discs is at most $8 |\chi(F)|/ (\mu/16)^2$ by Lemma \ref{MuCover}. The number of  tetrahedra of $T'$ that intersect such a disc is at most $k$, by Lemma \ref{Lem:TetWithinBall}, where $k$ depends on 
$\theta''''$, $a''''$, $b''''$.
By Proposition \ref{NormalDisk} \and Definition \ref{SufficientlyTransverse}, the intersection between each such disc and each tetrahedron lies in at most one normal triangle or square. Hence, the number of normal discs in $F$ is at most $2048 k |\chi(F)|/ \mu^2 = N|\chi(F)|$, for a constant $N$ that depends on 
$\theta''''$, $a''''$, $b''''$ and $\mu$. \end{proof}

Note that because $T'$ is $(\theta'''', a''''\mu, b''''\mu)$-thick, each tetrahedron of $T'$ has volume at least $v \mu^3$ by Lemma \ref{Lem:VolumeTet}, and so the number of tetrahedra of $T'$ is at most $\mathrm{volume}(M)/(v\mu^3)$.

Assume the surface $F$ satisfies the hypothesis of Proposition \ref{DiskNumber}. Thus, we may also now assume that $F$ intersects $T'$ in at most $N |\chi(F)|$ normal discs. We now wish to control the number of possibilities for $F$ within the complement of $T'$. This is a 3-dimensional submanifold $M_{\textrm{thin}}$ of $M$ that is isotopic to a union of components of $M_{(0,\mu/2)}$. Hence, by the Margulis lemma, it is a collection of copies of $T^2 \times (1,\infty)$ and solid tori.

\begin{lemma}
We may isotope $F$ so that the following hold:
\begin{enumerate} 
\item $F \cap M_{\textrm{thin}}$ is incompressible in $M_{\textrm{thin}}$;
\item $F \cap T'$ consists of at most $N |\chi(F)|$ normal discs.
\end{enumerate}
\end{lemma}

\begin{proof}
Suppose that $F \cap M_{\textrm{thin}}$ has a compression disc $D$ in $M_{\textrm{thin}}$. Then $\partial D$ bounds a disc $D'$ in $F$.  We may replace $D'$ by $D$ by an isotopy. This removes components of $F \cap T'$. The remainder of $F \cap T'$ therefore still consists of at most $N |\chi(F)|$ normal discs. We may repeat this process, reducing the number of normal discs, until $F \cap M_{\textrm{thin}}$ is incompressible \end{proof}

We therefore now assume that $F$ satisfies the conclusions of the above lemma. Note that $F$ is no longer necessarily minimal.
As explained in Section \ref{Sec:Overview}, the number of possibilities for $F$ within $T'$ is at most

\begin{equation}\tag{$*$}
\mathcal{F}=\left(
\begin{matrix}
7|T| + N|\chi(F)| \\
N|\chi(F)| 
\end{matrix}
\right),
\end{equation}
or alternatively $O(\mathrm{vol}(M))^{N |\chi(F)|}$.
We therefore have control over the number of ways $F$ intersects $T'$. 

We now wish to control the number of possibilities for $F$ in $M_{\textrm{thin}}$, the complement of $T'$ in $M$. 
Suppose that we have fixed a normal representative for $F$ within $T'$.
Since $F \cap M_{\textrm{thin}}$ is incompressible and orientable, it consists of properly embedded discs, boundary parallel annuli and boundary parallel tori. Tori are ruled out by our hypothesis that $F$ is essential in $M$. The position of the discs are determined up to isotopy by the curves $F \cap \partial M_{\textrm{thin}}$. We now consider the annuli.

\begin{lemma}
\label{Lem:AnnuliInTorusTimesIOneBoundary}
Let $C_1, \dots, C_{2n}$ be a collection of disjoint essential simple closed curves in $T^2 \times \{ 0 \}$. Then the number of collections of disjoint properly embedded annuli in $T^2 \times [0,1]$ with boundary equal to these curves, up to isotopy fixing the boundary, is $\left ( \begin{smallmatrix}
2n \\ n \end{smallmatrix} \right )$.
\end{lemma}

\begin{proof}
There is a product structure $S^1 \times S^1 \times [0,1]$ on $T^2 \times [0,1]$ so that each $C_i$ is of the form $S^1 \times p_i$ for some point $p_i$ on $S^1 \times \{ 0 \}$. A collection of properly embedded annuli with boundary equal to $C_1 \cup \dots \cup C_{2n}$ can be isotoped, keeping their boundary fixed, to the form $S^1 \times \alpha_q$ for arcs $\alpha_q$ in $S^1 \times [0,1]$. Thus, we need to count the number of properly embedded arcs in the annulus $S^1 \times [0,1]$, with boundary equal to given points $p_1, \dots, p_{2n}$ on the boundary. We claim that the number of such arcs, up to isotopy fixing the boundary, is $\left ( \begin{smallmatrix}
2n \\ n \end{smallmatrix} \right )$.

Counting arcs in such a system corresponds to counting matching pairs of parentheses, where each arc corresponds to a matched pair, and each point $p_i, i=1, ..., 2n$,  corresponds to an opening or closing parenthesis. Specifically, each arc in the collection is parallel to an interval in $S^1 \times \{ 0 \}$, and we assign opening and closing parentheses to the endpoints of this interval. Note that the points $p_1, \dots, p_{2n}$ are in a circular arrangement (rather than being arranged in a line, as in the classical problem), and therefore non-matching pairs cannot occur.  The standard combinatorial argument counts all options as $\left ( \begin{smallmatrix}
2n \\ n \end{smallmatrix} \right )$, which corresponds to designating $n$ points out of $2n$ as open parentheses, and the rest as closed ones. \end{proof}

\begin{lemma}
\label{Lem:AnnuliInSolidTorus}
Let $C_1, \dots, C_{2n}$ be a collection of disjoint essential non-meridional simple closed curves in the boundary of a solid torus. Then the number of collections of disjoint properly embedded annuli in the solid torus with boundary equal to these curves, up to isotopy fixing the boundary, is at most 
$\left ( \begin{smallmatrix}
2n \\ n \end{smallmatrix} \right )$.
\end{lemma}
 
\begin{proof}
There is a Seifert fibration of the solid torus in which these curves $C_1, \dots, C_{2n}$ are fibres. The base of this Seifert fibration is a disc, possibly with a single singular point that corresponds to an exceptional fibre. The curves project to points $p_1, \dots, p_{2n}$ in the disc boundary. The annuli are the inverse images of arcs in this disc, with endpoints $p_1, \dots, p_{2n}$. These arcs miss the singular point, if there is one. If there is a singular point, remove it from the disc, forming an annulus $S^1 \times [0,1]$. If there is no singular point, remove a regular point disjoint from the arcs. Then we obtain arcs in $S^1 \times [0,1]$ with $2n$ specified boundary points. The number of such arcs is $\left ( \begin{smallmatrix}
2n \\ n \end{smallmatrix} \right )$, as we showed in the proof of Lemma \ref{Lem:AnnuliInTorusTimesI}. \end{proof}

In the case where the curves are not longitudes, this upper bound is sharp. However, when they are longitudinal, it might not be.

\begin{lemma}
\label{Lem:AnnuliInTorusTimesI}
Let $C_1, \dots, C_{n}$ be a collection of disjoint essential simple closed curves in $T^2 \times \{ 0 \}$. Then the number of collections of disjoint properly embedded annuli in $T^2 \times [0,1]$, such that no annulus is parallel to an annulus in $T^2 \times \{ 1 \}$, and that intersects $T^2 \times \{ 0 \}$ in precisely $C_1 \cup \dots \cup C_{n}$, up to isotopy fixing $T^2 \times \{ 0 \}$, is at most $3^n$.
\end{lemma}

\begin{proof}

Denote such a collection of annuli by $\mathcal{C}$. Each annulus in $\mathcal{C}$ either runs from $T^2 \times \{ 0 \}$ to $T^2 \times \{ 1 \}$ or has both boundary components in $T^2 \times \{ 0 \}$. We first consider $k$ annuli of the latter type, and, as their boundary, pick the $2k$ components out of $C_1, \dots, C_{n}$. There are $\left ( \begin{smallmatrix}
n \\ 2k \end{smallmatrix} \right )$ ways to do this.  By Lemma \ref{Lem:AnnuliInTorusTimesIOneBoundary}, there are
$\left ( \begin{smallmatrix}
2k \\ k \end{smallmatrix} \right )$ ways of joining these with annuli, up to isotopy fixing $T^2 \times \{ 0 \}$. For some of these
choices, there is no way to extend this collection of $k$ annuli to a collection $\mathcal{C}$ of disjoint properly embedded annuli that intersect $T^2 \times \{ 0 \}$
precisely in $C_1 \cup \dots \cup C_{n}$. This happens when some $C_i$ is separated from $T^2 \times \{ 1 \}$
by one of the inserted annuli. However, when one can extend the collection of $k$ annuli to $\mathcal{C}$, 
there is a unique way to do this. Hence, the number of possibilities for such $\mathcal{C}$ is at most
$$\sum_{k = 0}^{\lfloor n/2 \rfloor} 
\left(
\begin{matrix}
n \\ 2k
\end{matrix}
\right)
\left(
\begin{matrix}
2k \\ k
\end{matrix}
\right).
$$
Each summand is the coefficient of $x^{n-2k}y^k z^k$ in $(x+y+z)^n$. Hence, the sum is at most $3^n$.
\end{proof}

We now complete the proof of Theorem \ref{Thm:NumberOfSurfaces}. Recall that we have arranged for the surface $F$ to intersect $T'$ in at most $N |\chi(F)|$ normal triangles and squares, where $N$ is a universal constant, and from this obtained an upper bound $(*)$ on the number of
possibilities for $F \cap T'$. Let us consider one such possibility. Then $F \cap \partial M_{\textrm{thin}}$ is a normal 1-manifold consisting of at most $4N |\chi(F)|$ normal arcs. The curves that bound discs in $M_{\textrm{thin}}$ bound discs of $F \cap M_{\textrm{thin}}$, and there is no choice, up to isotopy, for these discs. By Lemma \ref{Lem:AnnuliInTorusTimesI} and \ref{Lem:AnnuliInSolidTorus}, the number of possibilities for the annuli of $F \cap M_{\textrm{thin}}$ is at most 
$3^{4N |\chi(F)|}$. Hence, using $(*)$, the total number of possibilities for $F$ up to isotopy is at most $\mathcal{F} \cdot 3^{4N |\chi(F)|}$.

So, the number of possibilities for $F$ is at most $(c_1 \mathrm{vol}(M))^{c_2 |\chi|}$ for universal constants $c_1$ and $c_2$. 
This completes the proof of Theorem \ref{Thm:NumberOfSurfaces}.

\begin{proof}[Proof of Corollary \ref{Cor:Links}]
Let $K$ be a hyperbolic link in the 3-sphere. The exterior of $K$ has a triangulation using at most $t = 8c(K)$ tetrahedra, and hence, by a result of Gromov and Thurston \cite{Thurston}, the hyperbolic volume of $S^3 \setminus K$ is at most $v_3 t$, where $v_3$ is the volume of a regular ideal hyperbolic tetrahedron. Theorem \ref{Thm:NumberOfSurfaces} gives that the number of non-isotopic properly embedded orientable essential surfaces with Euler characteristic at least $\chi$ is no more than $(8 v_3 c_1 c(K))^{c_2 |\chi|}$. Setting $c'_1 = 8 v_3 c_1$ completes the proof.\end{proof}

\begin{restate}{Corollary}{Cor:notorientable}
The statement of Theorem \ref{Thm:NumberOfSurfaces} extends to non-orientable surfaces that are $\pi_1$-injective and boundary-$\pi_1$-injective.
\end{restate}

\begin{proof}
For such a surface $S$, one can consider $\tilde S$, the frontier of a regular neighbourhood of $S$. When $S$ is $\pi_1$-injective and boundary-$\pi_1$-injective, then $\tilde S$ is incompressible and boundary-incompressible. As $\tilde S$ is orientable, Theorem \ref{Thm:NumberOfSurfaces} bounds the number of possibilities for $\tilde S$ up to isotopy. This then bounds the number of possibilities for $S$, up to isotopy, since any given $\tilde S$ can arise from at most two possible $S$. \end{proof}

\section{Appendix}\label{Sec:Appendix}

In this appendix, we give an outline of the proof of the following result. 

\begin{reptheorem}{Thm:ExtensionBreslin}{\it
There are real numbers $0 < a \leq b \leq 1/40$ and $\theta > 0$ with the following property.
Let $\mu$ be any positive real number less than the 3-dimensional Margulis constant, and let $M$ be a complete finite-volume hyperbolic 3-manifold. Then there is a  $(\theta, a\mu,b\mu)$-thick triangulation of a subset of $M$ isotopic to $M^{\textrm{fat}}_{\mu/2}$ and lying in $M_{[\mu/4,\infty)}$}.
\end{reptheorem}

As stated in Section \ref{Sec:Triangulate}, this is a refined version of the following result of Breslin for manifolds with finite volume.

\begin{restate}{Theorem}{BreslinTh}
Let $\mu$ be a positive real number less than a 3-dimensional Margulis constant. There exist positive constants $A := A(\mu), B := B(\mu)$, and $\theta := \theta(\mu)$ with the following property. Any complete hyperbolic 3-manifold $M$ has a triangulation such that each of its tetrahedra lying in 
the thick part $M_{[\mu/2, \infty)}$ is $(\theta, A,B)$-thick.
\end{restate}

Note that several other authors have also constructed geodesic triangulations of hyperbolic 3-manifolds with a controlled number of tetrahedra. The initial construction was due to Thurston \cite{Thurston}, but the argument there was only sketched. Full details were given by Kobayashi and Rieck \cite{KobayashiRieck}, who paid particular attention to the case where the thin part of the manifold is non-empty. Similarly, our argument is most complicated when the thin part is non-empty. We initially avoid this case, as follows.

\subsection{Empty thin part} We will now give a proof of Theorem \ref{Thm:ExtensionBreslin} in the case where $M^{\textrm{fat}}_{\mu/2}$ is all of $M$, or equivalently when the thin part $M_{(0,\mu/4)}$ is empty. In particular, in this subsection, $M$ is closed.

We start by sketching Breslin's argument.

Breslin picks a collection $\calS$ of points in $M$ with the following properties:
\begin{enumerate}
\item for any $p \in M$, the ball $B(p, \mathrm{inj}(M,p)/5)$ centred at $p$ with radius $\mathrm{inj}(M,p)/5$ has a point of $\calS$ in its interior;
\item any point in $\calS \cap M_{[\mu, \infty)}$ has distance at least $\mu/100$ from every other point in $\calS$;
\item any $p \in M_{[\mu, \infty)}$ has distance less than $\mu/100$ from some point in $\calS$. 
\end{enumerate}

Breslin then ensures that there is a Delaunay triangulation of $M$ associated with $\calS$.
This triangulation has vertex set $\calS$ and is such that four points of $\calS$ are the vertices of a tetrahedron 
if and only if there is some point $p$ in $M$ that is equidistant from these four points and there is no other point of $\calS$ closer to $p$. 
Condition (1) above ensures that each point $p$ in $M$ has distance at most $\mathrm{inj}(M,p)/5$ from some point in $\calS$. 
This is less than the injectivity radius of $M$ at $p$, and so the points in $\calS$ closest to $p$ lie in a ball around $p$ that is isometric 
to a ball in $\mathbb{H}^3$. Hence, one can view $p$ as the centre of the circumsphere of these four points.

Breslin needs to assume Condition (1) above, as well the following  `genericity' hypothesis, to ensure that
the Delaunay triangulation is defined. The triangulation is dual to the Voronoi domain for $\calS$. Recall that this is a cell complex, where each 3-cell surrounds a point $p$ in $\calS$. This 3-cell is defined to be the set of points in $M$ that are at least as close to $p$ as to any other point in $\calS$. For a `generic' choice of $\calS$, one would expect that at most four 3-cells intersect at any point. When this is the case, the 0-cells in the interior of $M$ are exactly where four 3-cells meet. These four 3-cells surround four points of $\calS$ which span a tetrahedron of the Delaunay triangulation. The 0-cell is at the centre of the circumsphere for these four vertices. One can then see why no other point of $\calS$ can lie within this circumsphere.

Breslin then perturbs $\calS$ to a new set $\calS'$, by moving each point in $\calS \cap M_{[\mu, \infty)}$ a distance at most $\mu/1000$. 
The required triangulation of $M$ is the Delaunay triangulation associated with $\calS'$. Breslin makes this perturbation
to ensure that the resulting tetrahedra in $M_{[\mu, \infty)}$ are $(\theta, A,B)$-thick. He therefore embarks upon a detailed analysis of hyperbolic tetrahedra that are not thick. The tetrahedra considered here have edge lengths between $\mu/100$ and $\mu/50$. In fact, because one wants to perturb the vertices a little, one should relax these conditions a little; Breslin sets $A = (4/5)\mu/100$ and $B = (11/10)\mu/50$. The difficulty is that these constraints alone do not guarantee a positive lower bound on the interior angles of the tetrahedra. However, the tetrahedra in $M_{[\mu, \infty)}$ arising in the Delaunay construction have an important extra property: the distance from the centre of the circumsphere to any vertex of the tetrahedron is at most $(11/10)\mu/100$. So Breslin considers tetrahedra with an upper bound $R = (11/10)\mu/100$ on their circumradius. Again this upper bound on circumradius does not force a positive lower bound on the dihedral angles. However, Breslin shows that the situations where there is a dihedral angle that is very small are extremely limited.

The prototypical example of a tetrahedron with edge lengths between fixed constants $A$ and $B$, and with circumradius at most $R$, but with some small dihedral angles, is as follows. Consider a sphere of radius $R$ in $\mathbb{H}^3$. Pick some equator for this sphere. Now place all four vertices of the tetrahedron near this equator. (See Figure \ref{Fig:SliverTetrahedron}.) As long as no two vertices are too close, one can guarantee a uniform lower bound on edge lengths. The edge lengths are bounded above by $2R$. However, clearly there is no positive lower bound on the dihedral angles. Indeed if all four vertices lay exactly on the equator, their convex hull would lie in a plane and hence some dihedral angles would be zero.
Tetrahedra such as the above are informally called `slivers', although this term is used in a precise way below and in Breslin's paper.

\begin{figure}
  \includegraphics[width=0.4\textwidth]{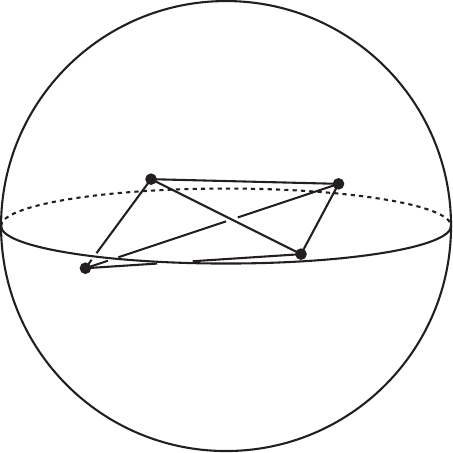}
  \caption{A tetrahedron with four vertices near the equator of a sphere forming a sliver.}
  \label{Fig:SliverTetrahedron}
\end{figure}

Breslin shows that when some dihedral angle is small, then each of the four vertices must lie close to the circumcircle specified by the other three. He shows that this set of points close to the circumcircle has very small volume. Hence, one can perturb one of the points to avoid this set, and the resulting tetrahedron is no longer a sliver. Thus, he perturbs $\calS$ one point at a time. Of course, one needs to ensure that in this perturbation, one does not destroy the progress one has made previously, by reverting a previous tetrahedron to a sliver. However, the fact that the volume of the bad set of points is small ensures that one can perturb each point so that no tetrahedron to which it is incident becomes a sliver. In this way, Breslin creates the set $\calS'$ and each tetrahedron of its Delaunay triangulation lying in $M_{[\mu,\infty)}$ is then $(\theta, A,B)$-thick for some universal positive $\theta$. 

\begin{remark}
We briefly note here the background that Breslin's proof relied on. Breslin needed to know that a given set of points in a Riemannian manifold can be perturbed so that it forms the vertex set of a Delaunay triangulation, and Breslin proved this using a theorem of Leibon and Letscher \cite{LeibonLetscher}. However, the result of Leibon and Letscher was later shown to be incorrect \cite{Boissonnat3}. But there is now an alternative way to supply the result needed by Breslin, using work of Boissonnat, Dyer and Ghosh \cite{Boissonnat}. This assumes that the initial set $\calS$ of points in the manifold is a \emph{$(\tau_0, \epsilon)$-net}, which means any point in the manifold has distance less than $\epsilon$ from some point in $\calS$ and any two points in $\calS$ are at least $\tau_0\epsilon$ apart. Theorem 3 in \cite{Boissonnat} implies that if $\calS$ is a $(\tau_0, \epsilon)$-net, with $\epsilon$ less than a quarter of the injectivity radius of the manifold and also less than an explicit function of the sectional curvatures of the manifold, its dimension and $\tau_0$, then there is a perturbation of $\calS$ to a set of points that is the vertex set of a Delaunay triangulation of the manifold. Thus, when $M$ is a closed hyperbolic 3-manifold, a small enough $\epsilon$ can be chosen, and any $(1,\epsilon)$-net then satisfies the hypotheses of Theorem 3 in \cite{Boissonnat}. Moreover, there is always a $(1,\epsilon)$-net
obtained by picking a maximal set $\calS$ of points in $M$, no two of which are less than $\epsilon$ apart. Thus for closed hyperbolic 3-manifolds with a lower bound on injectivity radius, Theorem 3 in \cite{Boissonnat} can be used instead of \cite{LeibonLetscher} to complete Breslin's proof. However, for a hyperbolic 3-manifold $M$ with cusps, its injectivity radius is zero, and so a positive $\epsilon$ satisfying the hypotheses of Theorem 3 in \cite{Boissonnat} cannot be found. Nonetheless, the reasoning  behind Theorem 3 in \cite{Boissonnat} can be reproduced. This is because the arguments in \cite{Boissonnat} are all `local' in the sense that when a point in $\calS$ is perturbed, the set of permitted perturbations only depends on the set of `nearby' points in $\calS$. (See, for example, Remark 2 in \cite{DVW}.) Hence, it also applies when $\calS$ is a \emph{generalised $(\tau_0, \epsilon)$-net}, which means that each point $p$ in $M$ has distance less than $\min \{ \epsilon, \textrm{inj}(M,p)/5 \}$ from some point in $\calS$, and for each $s$ in $\calS$, the distance from $s$ to any other point in $\calS$ is at least $\tau_0 \min \{ \epsilon, \textrm{inj}(M,s)/10 \}$. Note that a generalised $(\tau_0, \epsilon)$-net does exist in $M$ for $\tau_0 \leq 1$: pick a maximal collection $\calS$ of points in $M$ such that for any two points $s$ and $s'$ in $\calS$, the distance between them is at least $\tau_0 \min \{ \epsilon, \textrm{inj}(M,s)/5, \textrm{inj}(M,s')/5 \}$. Hence, perturbing $\calS$ as in \cite{Boissonnat}, we obtain the vertex set of a Delaunay triangulation of $M$.
\end{remark}

We now explain Breslin's general plan in more detail, as well as the modifications that we need to make to it. The following definition is due to Breslin.

\begin{definition}
Let $\Delta$ be a geodesic tetrahedron in $\mathbb{H}^3$ and let $p$ be a vertex of $\Delta$. Define the following quantities:
\begin{enumerate}
\item let $R_\Delta$ be the circumradius of $\Delta$;
\item let $\ell_\Delta$ be the length of the shortest edge of $\Delta$;
\item let $c_p$ be the circumradius of the face opposite $p$;
\item let $d_p$ be the distance from $p$ to the plane defined by the face opposite to $p$.
\end{enumerate}
For $\sigma, \rho > 0$, $\Delta$ is called a \emph{$(\sigma,\rho)$-sliver} if $d_p / c_p \leq \sigma$ and $R_\Delta / \ell_\Delta \leq \rho$ for some vertex $p$ of $\Delta$.
\end{definition}

In the situation that we and Breslin consider, $\ell_\Delta$ is between $A = (4/5)\mu/100$ and $B = (11/10)\mu/50$. Furthermore, 
$R_\Delta$ is at most $B = (11/10)\mu/50$ since the tetrahedra occur in a Delaunay triangulation. Also, $R_\Delta$ is at least $A/2 = (2/5)\mu/100$.
So, $R_\Delta / \ell_\Delta$ is bounded above and below by universal constants. Hence, the quantity $\rho$ in the above definition can be fixed as some constant and then the second condition $R_\Delta / \ell_\Delta \leq \rho$ is automatically satisfied. It is therefore $\sigma$ that plays the more crucial role. 

The following (Lemma 2.1 in \cite{Breslin}) elucidates the nature of slivers further. It can be viewed as saying that if the tetrahedron that we are considering has some small dihedral angle, then it must be a $(\sigma,\rho)$-sliver for some small $\sigma$.

\begin{lemma}
\label{Lem:Breslin2.1}
If a geodesic tetrahedron $\Delta$ in $\mathbb{H}^3$ has edge lengths between $A$ and $B$, and circumradius at most $R$, and it has some dihedral angle less than
$$\arcsin \left ( \frac{\sinh(\sigma A/2)}{\sinh(B)} \right)$$
for some $\sigma > 0$, then $\Delta$ is a $(\sigma,R/A)$-sliver.
\end{lemma}

The next lemma (Lemma 2.2 in \cite{Breslin}) in Breslin's paper is concerned with a lower bound for the altitudes of a geodesic triangle in $\mathbb{H}^2$. We restate it slightly so that the nature of this lower bound is clarified when the edge lengths are all small.

\begin{lemma}
Let $\mu > 0$ be bounded above by some constant. Then
a geodesic triangle in $\mathbb{H}^2$ with edge lengths between $a \mu$ and $b \mu$ and circumradius at most $r \mu$ has altitudes bounded from below by $y(a,b,r) \mu$, where $y(a,b,r) >0$.
\end{lemma}

\begin{proof}
Breslin first considers the altitudes of triangles with edge lengths at least $a \mu$, with circumradius at most $r \mu$, and where the altitude ends on an edge of the triangle (versus the extended edge outside of the triangle). Breslin argues that the smallest such altitude is minimised in the following geometric arrangement. Start with a circle of radius $r \mu$ in the hyperbolic plane. Place one vertex $p$ of this triangle on the circle, and place the two other vertices $q$ and $r$ on each side of it on the circle, both at hyperbolic distance $a \mu$ from $p$. The relevant altitude is the one from $p$. It is clear that this is positive for any given $a$, $b$, $r$, and $\mu$, but we must investigate the length of this altitude as $\mu \rightarrow 0$. So we place the centre of the circle at the centre of the Poincar\'e disc model, and then perform a Euclidean dilation of the circle and triangle of factor $1/\mu$. Then as $\mu \rightarrow 0$, the triangle increasingly approximates a fixed Euclidean triangle, and the Euclidean length of its altitude is bounded below by a quantity $Y(a,b,r) > 0$. Hence when we reverse the dilation, the Euclidean length of the altitude of the original triangle is bounded below by $Y(a,b,r) \mu$. Hyperbolic lengths in the Poincar\'e disc model are at least Euclidean lengths, and so we have the same lower bound on the hyperbolic length of the altitude. 

Then Breslin considers the case of a general altitude for a triangle with side lengths between $a\mu$ and $b\mu$ and circumradius at most $r \mu$. He shows that the length of such an altitude is at least
$$\arcsinh \left ( \frac{\sinh (a\mu)}{\sinh (b \mu)} \sinh (Y(a,b,r) \mu) \right ).$$
As $\mu \rightarrow 0$,
$$\arcsinh \left ( \frac{\sinh (a\mu)}{\sinh (b \mu)} \sinh (Y(a,b,r) \mu)  \right ) \sim \frac{a}{b} Y(a,b,r) \mu.$$
In particular, this quantity is bounded below by $y(a,b,r) \mu$ for some positive constant $y(a,b,r)$ depending on $a$, $b$ and $r$.
\end{proof}

The following is Lemma 2.3 in \cite{Breslin}. It deals with $(\sigma,R/A)$-slivers $\Delta$. By definition, \emph{some} vertex $p$ of $\Delta$ satisfies $d_p / c_p \leq \sigma$. But the lemma says that for \emph{every} vertex $v$ of $\Delta$, $d_v / c_v$ can be bounded above in terms of $\sigma$, $A$, $B$ and $R$.

\begin{lemma}
\label{Lem:Lemma3Breslin}
If a geodesic tetrahedron $\Delta$ in $\mathbb{H}^3$ with edge lengths between $a \mu$ and $b \mu$ and circumradius at most $r \mu$ is a $(\sigma, r/a)$-sliver for some $\sigma > 0$, then $d_v / c_v$ is bounded above by a constant $N:= N(\sigma, a, b, r)$ for each vertex $v$ of $\Delta$. Moreover, $N$ can be chosen so that $N \rightarrow 0$ as $\sigma \rightarrow 0$ and $a,b,r$ remain fixed and $\mu$ is bounded above.
\end{lemma}

\begin{proof} 
Breslin shows that $d_v / c_v$ is bounded above by a constant $n:= n(\sigma, a\mu, b\mu, r \mu)$. We must examine how $n$ behaves as $\mu \rightarrow 0$. He sets
$$
n(\sigma,a\mu, b\mu, r\mu) = \frac{2}{a\mu} \arcsinh \left ( \frac{\sinh (b\mu) \sinh(\sigma r \mu)}{\sinh (y(a,b,r) \mu)} \right).
$$
As $\mu \rightarrow 0$,
$$
n(\sigma,a\mu, b\mu, r\mu) \sim \frac{2}{a\mu} \arcsinh \left ( \frac{b \sigma r \mu}{y(a,b,r)} \right)
\sim \frac{2b \sigma r}{a y(a,b,r)}.
$$
Hence, $n(\sigma,a\mu, b\mu, r\mu)$ is bounded above by a constant $N(\sigma, a, b, r)$. Furthermore, this tends to zero as  $\sigma \rightarrow 0$ and $a,b,r$ remain fixed and $\mu$ is bounded above. \end{proof}

In the following result, which is Lemma 2.4 in \cite{Breslin}, the location of each vertex in a sliver is controlled in terms of the other three vertices. Specifically, it has to be close to the circumcircle of the other three vertices. This confirms a key feature of the examples given in Figure \ref{Fig:SliverTetrahedron}, where the circumcircle of any three vertices is approximately the equator of the sphere.

\begin{lemma}
\label{Lem:Lemma4Breslin}
If a geodesic tetrahedron $\Delta$ in $\mathbb{H}^3$ with edge lengths between $a \mu$ and $b \mu$ and circumradius at most $r \mu$ is a $(\sigma, r/a)$-sliver for some $\sigma > 0$, then the distance of any vertex of $\Delta$ from the circumcircle of the remaining three vertices is at most $K = k(\sigma, a, b, r) \mu$ for a constant $k$. Moreover, $k(\sigma, a, b, r)$ can be chosen so that $k(\sigma, a, b, r) \rightarrow 0$ as $\sigma \rightarrow 0$ and $a,b,r$ remain fixed and $\mu$ is bounded above.
\end{lemma}

\begin{proof}
Breslin's Lemma 2.4 states that the distance of any vertex of $\Delta$ from the circumcircle of the remaining three vertices is at most $K = K(\sigma, a\mu, b\mu, r\mu)$ for a function $K$. We must examine how $K$ behaves as $\mu$ varies. The function $K$ is defined by Breslin as follows:
$$
K(\sigma, a\mu, b\mu, r\mu) = 
$$
$$n(\sigma, a\mu, b \mu, r \mu) r \mu + 
\frac{2 \tan (J(\sigma, a\mu, b\mu, r\mu)) \cosh(r\mu) \sinh(r \mu)}
{\sinh (a\mu/2)}.
$$
Here,
$$
J(\sigma, a\mu, b\mu, r\mu) = \pi - 4 \arctan(e^{-n(\sigma, a\mu, b \mu, r \mu) r \mu}).
$$
Now, as $\mu \rightarrow 0$ and $a,b,r$ remain fixed, then $e^{-n(\sigma, a\mu, b \mu, r \mu) r \mu}$ is to first order
$$1 -(n(\sigma, a\mu, b \mu, r \mu) r \mu),$$
Moreover, near $x = 1$, $\arctan(x)$ is $\pi/4 + (x-1)/2$ to first order. So, to first order, $4 \arctan(e^{-n(\sigma, a\mu, b \mu, r \mu) r \mu})$ is
$$\pi - 2 n(\sigma, a\mu, b \mu, r \mu) r \mu.$$
and therefore, to first order, $J( \sigma, a\mu, b\mu, r\mu)$ is 
$$2n(\sigma, a\mu, b \mu, r \mu) r \mu.$$
So, as $\mu \rightarrow 0$,
\begin{align*}
\frac{2 \tan (J(\sigma, a\mu, b\mu, r\mu)) \cosh(r\mu) \sinh(r \mu)}{\sinh (a\mu/2)} 
&\sim \frac{4n(\sigma, a\mu, b \mu, r \mu) r^2 \mu^2}{a\mu/2} \\
& = 8 n(\sigma, a\mu, b \mu, r \mu) (r^2/a) \mu.
\end{align*}
So,
\begin{align*}
K(\sigma, a\mu, b\mu, r\mu) &\sim n(\sigma, a\mu, b \mu, r \mu) (r + 8r^2/a) \mu \\
&\leq N(\sigma, a, b , r) (r + 8r^2/a) \mu.
\end{align*}
So,
$K/\mu$ is bounded above by a function $k(\sigma, a,b,r)$ and moreover, by Lemma \ref{Lem:Lemma3Breslin}, this tends to zero as $\sigma\rightarrow 0$ and $a,b,r$ remain fixed and $\mu$ is bounded above. \end{proof}

The following is from Breslin \cite[Section 2]{Breslin}.

\begin{definition}
Let $T$ be a geodesic triangle in $\mathbb{H}^3$ with edge lengths in $[A,B]$ and circumradius at most $R$, and let $\sigma > 0$. Then the \emph{$(\sigma, A, B, R)$-sliver region} of $T$ is the set of points $p$ in $\mathbb{H}^3$ such that the tetrahedron spanned by $p$ and the vertices of $T$ is a $(\sigma, R/A)$-sliver with edge lengths in $[A,B]$ and circumradius at most $R$.
\end{definition}

The following is a version of Lemma 2.5 in \cite{Breslin}.

\begin{lemma}
\label{Lem:Lemma5Breslin}
Let $T$ be a geodesic triangle in $\mathbb{H}^3$ with edge lengths in $[a\mu, b\mu]$ and circumradius at most $r\mu$. Then the volume of the $(\sigma, a\mu, b\mu, r\mu)$-sliver region of $T$ is at most $v(\sigma, a, b, r) \mu^3$ for some continuous function $v(\sigma, a, b, r)$. Moreover, $v(\sigma, a, b, r) \rightarrow 0$ as $\sigma \rightarrow 0$ and $a,b,r$ remain fixed and $\mu$ is bounded above.
\end{lemma}

\begin{proof}
By Lemma \ref{Lem:Lemma4Breslin}, the sliver region is contained within the $K$-neighbourhood of the circumcircle of $T$, for $K = k(\sigma, a, b, r) \mu$. The radius of the circumcircle is at most $r \mu$. Hence, because $\mu$ is bounded above, the volume of this $K$-neighbourhood is at most $v(\sigma, a, b, r) \mu^3$. Since $k(\sigma, a, b, r) \rightarrow 0$ as $\sigma \rightarrow 0$ and $a,b,r$ remain fixed and $\mu$ is bounded above, we deduce that $v(\sigma, a, b, r) \rightarrow 0$ as $\sigma \rightarrow 0$ and $a,b,r$ remain fixed and $\mu$ is bounded above. \end{proof}

As explained above, Breslin picked a collection of points $\calS$ in $M$ with the following three properties:
\begin{enumerate}
\item for any $p \in M$, the ball $B(p, \mathrm{inj}(M,p)/5)$ centred at $p$ with radius $\mathrm{inj}(M,p)/5$ has a point of $\calS$ in its interior;
\item any point in $\calS \cap M_{[\mu, \infty)}$ has distance at least $\mu/100$ from every other point in $\calS$;
\item any $p \in M_{[\mu, \infty)}$ has distance less than $\mu/100$ from some point in $\calS$.
\end{enumerate}

Breslin then perturbs $\calS$ to a set $\calS'$. Following Breslin, we say that $\calS'$ is a \emph{good perturbation} of $\calS$ if each point is moved a distance at most $\mu/1000$ and furthermore $\calS$ and $\calS'$ are the vertex sets of Delaunay triangulations. We refer to the perturbation of a point $p$ in $\calS$ by $p'$.

\begin{lemma}
\label{Lem:Lemma6Breslin}
Let $p \in \calS \cap M_{[\mu, \infty)}$. The number of triples $\{ q,r,s \}$ of points in $\calS$ such that $[p',q',r',s']$ is a tetrahedron in the Delaunay triangulation on some good perturbation $\calS'$ of $\calS$ is bounded above by an absolute constant.
\end{lemma}

\begin{proof}
In Lemma 3.1 of \cite{Breslin}, Breslin proves that the number of triples is bounded above by
$$\left ( \frac{\mathrm{vol}(B(\mu/50+\mu/500))}{\mathrm{vol}(B(\mu/200 - \mu/1000)} \right )^3.$$
Here, $B(R)$ is a ball of radius $R$ in $\mathbb{H}^3$. Since the volume of $B(R)$ is asymptotically $(4/3)\pi R^3$ as $R \rightarrow 0$, the number of triples is bounded above by an absolute constant. \end{proof}

The proof of Theorem \ref{Thm:ExtensionBreslin} in the case where $M = M^{\textrm{fat}}_{\mu/2}$ now follows that in Breslin's paper.
We pick a set $\calS$ of points in $M$ satisfying the following properties:
\begin{enumerate}
\item any two points in $\calS$ are at least $\mu/100$ apart;
\item any $p \in M$ has distance less than $\mu/100$ from some point in $\calS$. 
\end{enumerate}
We pick an ordering on  the points $\calS$, denoting them by $\{ p_1, \dots, p_k \}$. We will perturb these points in turn, and denote the perturbation of $p_i$ by $p'_i$. Let $\mathcal{S}_i$ be the result of perturbing the first $i$ points; so $\mathcal{S}_i = \{ p'_1, \dots, p'_i, p_{i+1}, \dots, p_k \}.$ We wish to perturb $p_{i+1}$ to $p_{i+1}'$. Let $\mathcal{U}_i$ be the set of triples $\{ q,r,s \} \subseteq \mathcal{S}_i \setminus \{ p_{i+1} \}$ such that $[p_{i+1}', q', r', s']$ is a tetrahedron in the Delaunay triangulation on some good perturbation of $\mathcal{S}$. By Lemma \ref{Lem:Lemma6Breslin}, $|\mathcal{U}_i|$ is bounded above by an absolute constant. We pick the perturbation $p'_{i+1}$ of $p_{i+1}$ so that it lies with in the ball of radius $\mu/1000$ about $p_i$ and so that it avoids the sliver regions of every triangle in $\mathcal{U}_i$. We also ensure that $\calS_i$ is generic in the sense that it is the vertex set of a Delaunay triangulation. No two points in $\calS_i$ are closer than $(4/500)\mu$ apart, and so this is a lower bound for the edge lengths of this Delaunay triangulation. Each point in $M$ has distance at most $(11/1000)\mu$ from some point in $\calS_i$. Hence, this is an upper bound for the circumradius of each tetrahedron in the triangulation. Each edge of each tetrahedron therefore has length at most $(11/500)\mu$.
The ball of radius $\mu/1000$ has volume at least $(4/3)\pi (\mu/1000)^3$. By Lemma \ref{Lem:Lemma5Breslin}, the volume of each sliver region is at most $v(\sigma, a, b, r) \mu^3$, where 
$$a = 4/500, \qquad b = 11/500, \qquad r = 11/1000.$$ 
Moreover, $v(\sigma, a, b, r) \rightarrow 0$ as $\sigma \rightarrow 0$. Hence if $\sigma$ is sufficiently small (i.e. less than some universal positive number), we can ensure that there is such a point $p'_{i+1}$. Repeating this for each point in $\calS$ gives the required set $\calS'$. Let $\mathcal{T}$ be the resulting Delaunay triangulation. Since none of its tetrahedra is a sliver, we deduce from Lemma \ref{Lem:Breslin2.1}, that there is a uniform positive lower bound $\theta$ on each of its dihedral angles. Hence, $\mathcal{T}$ is an $(a\mu, b\mu, \theta)$-thick triangulation of $M$. Note that $b = 11/500 < 1/40$. This completes the proof of Theorem \ref{Thm:ExtensionBreslin} in the case where $M_{(0,\mu/4)}$ is empty.

\subsection{The general case: when the thin part might be non-empty}
We now prove Theorem \ref{Thm:ExtensionBreslin} in the case where the thin part $M_{(0,\mu/4)}$ might be non-empty. In that case,  $M^{\textrm{fat}}_{\mu/2}$ has non-empty toral boundary.

The proof follows the same general approach as in the closed case. A collection of points $\mathcal{S}$ is picked in $M$ and a perturbation is then applied to it, giving a set $\mathcal{S}'$. The aim is to show that $\mathcal{S}'$ can be chosen so that it is the vertex set of a Delaunay triangulation that is thick in $M_{[\mu/4,\infty)}$. 
Crucially, unlike in Breslin's work, we also need to ensure that this triangulation contains a simplicial subset that is isotopic to $\partial M^{\textrm{fat}}_{\mu/2}$.
Some care must be taken with the initial choice of $\mathcal{S}$ near $\partial M^{\textrm{fat}}_{\mu/2}$. We therefore now establish various geometric properties of this boundary.

\begin{lemma}
\label{Lem:TubeRadius}
Let $X$ be a component of $\mathrm{cl}(M - M^{\textrm{fat}}_{\mu/2})$ that is an $r$-neighbourhood of a geodesic
with length less than $\mu/4$. Then $r \geq \mu/10$.
\end{lemma}

\begin{proof}
Let $\ell < \mu/4$ be the length of the geodesic. Consider any point on $\partial X$ and its lift to $\mathbb{H}^3$.
The covering transformation corresponding to $\ell$ translates this point distance $\ell \cosh r$ along the inverse image of 
$\partial X$, together with a rotation of some angle. Hence the distance in $\mathbb{H}^3$ between these two points
is at most $2r + \ell \cosh r$. But by definition of the boundary of $M_{[\mu/2,\infty)}$, this distance
is at least $\mu/2$. Hence,
$$2r + (\mu/4) \cosh r \geq 2r + \ell \cosh r \geq \mu/2.$$
Now suppose that $r$ is less than $\mu/10$. Since 
$\mu \leq 0.776$ \cite[Theorem 1.5(2)]{FPS}, we deduce that 
$\cosh r \leq \cosh 0.0776 < 1.004$. So,
$$2r \geq \mu \left(\frac{1}{2} - \frac{1.004}{4} \right)$$
and hence $r \geq (0.1245) \mu$, which is a contradiction. \end{proof}

\begin{lemma}
\label{Lem:BoundaryHasControlledGeometry}
The boundary of $M^{\textrm{fat}}_{\mu/2}$ inherits a Riemannian metric that is Euclidean with injectivity radius more than $\mu/4$.
\end{lemma}

\begin{proof}
The metric is Euclidean, because the isometry group of each component of $ \partial M^{\textrm{fat}}_{\mu/2}$ acts
transitively. The injectivity radius of the torus is half the length of a shortest closed geodesic. So
consider a closed geodesic in the torus, based at some point. This
geodesic lifts to a path in the inverse image of the torus in $\mathbb{H}^3$. If the endpoints of this path are distinct, then
they have hyperbolic distance at least $\mu/2$. Hence, their Euclidean distance is more than $\mu/2$, as required. On the other hand,
if the endpoints are the same, then the geodesic is a meridian of a solid toral component of $\mathrm{cl}(M - M^{\textrm{fat}}_{\mu/2})$. This forms the boundary of a circle with radius at least $r$.
Hence, by Lemma \ref{Lem:TubeRadius}, its length is at least $2 \pi \sinh r \geq 2 \pi r \geq (\pi/5) \mu > \mu /2$, as required. \end{proof}

The above lemma refers to the Riemannian metric on $\partial M^{\textrm{fat}}_{\mu/2}$. We will also need to consider the induced metric on
$\partial M^{\textrm{fat}}_{\mu/2}$ that it inherits as a subspace of $M$. The following lemma compares these metrics at small scales.

\begin{lemma}
\label{Lem:DistanceComparison}

For any $\epsilon > 0$, there is $\zeta \in (0,1/100)$ with the following property.
Let $d_{eucl}$ denote the Euclidean Riemannian metric on $\partial M^{\textrm{fat}}_{\mu/2}$, and let $d_{hyp}$ denote the
metric it inherits from $M$. Then for any two points $x$ and $y$ in $\partial M^{\textrm{fat}}_{\mu/2}$ with $d_{hyp}(x,y) \leq \zeta\mu$,
we have the following comparison:
$$d_{hyp}(x,y) \leq d_{eucl}(x,y) \leq \min \{ 1.12, (1 + \epsilon) \} d_{hyp}(x,y).$$

\end{lemma}

\begin{proof}

The fact that $d_{hyp}(x,y) \leq d_{eucl}(x,y)$ is automatic, since $d_{hyp}(x,y)$ is realised by a shortest geodesic in $M$ between $x$ and $y$, whereas $d_{eucl}(x,y)$ is realised by a shortest geodesic in $\partial M^{\textrm{fat}}_{\mu/2}$.

So we focus on the second inequality. 
Let $W$ be $\mathrm{cl}(N_{\zeta\mu}(\partial M^{\textrm{fat}}_{\mu/2}) - M^{\textrm{fat}}_{\mu/2})$. Thus, $W$ is a collar  on $\partial M^{\textrm{fat}}_{\mu/2}$ lying outside of $M^{\textrm{fat}}_{\mu/2}$ with width $\zeta\mu$. Because of Lemma \ref{Lem:TubeRadius}, $W$ does not reach as far as the core geodesic
of any solid toral components of $M - M^{\textrm{fat}}_{\mu/2}$. So $W$ is a disjoint union of copies of $T^2 \times [0,1]$. 

Let $\partial_- W$ be the tori
$\partial W - \partial M^{\textrm{fat}}_{\mu/2}$. Let $p \colon W \rightarrow \partial_- W$ be closest point retraction. Then $p$ does not increase the length of any curve. Moreover, $p$ scales the Riemannian metric on $\partial M^{\textrm{fat}}_{\mu/2}$ by a factor between $1$ and
$\sinh (\mu/10- \zeta\mu) / \sinh(\mu/10)$, for the following reason. Consider a solid toral component $Q$ of $\mathrm{cl}(M - M^{\textrm{fat}}_{\mu/2})$.
Let $r$ be the distance from the core geodesic of $Q$ to $\partial M^{\textrm{fat}}_{\mu/2}$. By Lemma \ref{Lem:TubeRadius}, $r \geq \mu/10$. The circumference
of a meridian circle of this tube with boundary $Q$ is $2\pi \sinh r$, and $p$ sends this meridian circle to a circle with circumference $2 \pi \sinh (r - \zeta\mu)$. The ratio $\sinh (r - \zeta\mu) / \sinh r$ is minimised when $r$ is as small as possible. Setting this minimum to be $\mu/10$ gives the required scale factor. One must also consider the scale factor in directions other than the meridian. And one must also consider the scale factor for cusp components of $M - M^{\textrm{fat}}_{\mu/2}$. But it is straightforward to check that the meridians for the solid tori, as above, gives the smallest scale factor. 

The scale factor $\sinh (\mu/10- \zeta\mu) / \sinh(\mu/10)$ tends to $1$ as $\zeta \rightarrow 0$. Hence, because $\mu$ is bounded above, the scale factor is at least $1/(1+\epsilon)$ if $\zeta$ is sufficiently small.
Also since $\mu \leq 0.776$ and $\zeta \leq 1/100$, the scale factor is at least $1/1.12$. \end{proof}

In the proof of Theorem \ref{Thm:ExtensionBreslin}, we will start with a triangulation of $\partial M^{\textrm{fat}}_{\mu/2}$. 
This will have edges that are geodesic in the Euclidean metric on $\partial M^{\textrm{fat}}_{\mu/2}$, 
bounding triangular faces.
We would like to ensure that when these edges and faces are straightened in $M$, the result is a simplicial subset of the triangulation of $M$.
To do so, we introduce the following concept.

Let $\calS$ be a set of points in $M$ that forms the vertex set of a Delaunay triangulation and such that $\calS \cap \partial M^{\textrm{fat}}_{\mu/2}$ forms the vertex set of a Delaunay triangulation of $\partial M^{\textrm{fat}}_{\mu/2}$ with its Euclidean Riemannian metric. For $\epsilon > 0$, we say that $\calS$ is \emph{$\epsilon$-robust} if for any triangle of the Delaunay triangulation of $\partial M^{\textrm{fat}}_{\mu/2}$ with vertices $x$, $y$, $z$, and Euclidean circumcentre $w$, the hyperbolic distance between $w$ and any point in $\calS \setminus \{ x,y,z \}$ is at least $(1+\epsilon)$ times the hyperbolic distance from $w$ to $x$, $y$ and $z$.

\begin{lemma} 
\label{Lem:RobustImpliesSimplicial}

Given $\epsilon \in (0,1)$, let $\zeta \in (0,1/100)$ be the constant from Lemma \ref{Lem:DistanceComparison}. Then there are $c \in (0,\zeta/2]$ and $\lambda \in (0,1)$ with the following property. Let $\calS$ be an $\epsilon$-robust set such that each edge of the Delaunay triangulation of $\partial M^{\textrm{fat}}_{\mu/2}$ with vertex set $\calS \cap \partial M^{\textrm{fat}}_{\mu/2}$ has Euclidean length in $[c\mu/4, c\mu/2]$. Suppose also that each point of 
$\partial M^{\textrm{fat}}_{\mu/2}$ has Euclidean distance less than $c\mu/2$ from some point of $\calS \cap \partial M^{\textrm{fat}}_{\mu/2}$.
Let $\calS'$ be a perturbation of $\calS$ in $M$, moving each point a hyperbolic distance at most $c\lambda\mu/1000$, and forming the vertex set of a Delaunay triangulation of $M$. 
Then for any three points in $\calS$ that span a triangle in the triangulation of 
$\partial M^{\textrm{fat}}_{\mu/2}$, their perturbations span a triangle in the Delaunay triangulation of $M$ with vertex set $\calS'$.

\end{lemma}

\begin{proof}

Consider a triangle $[x,y,z]$ of the Delaunay triangulation of $\partial M^{\textrm{fat}}_{\mu/2}$ with vertex set $\calS \cap \partial M^{\textrm{fat}}_{\mu/2}$. Let $w$ be its Euclidean circumcentre. Then $d_{eucl}(x,w) = d_{eucl}(y,w) = d_{eucl}(z,w)$. Also
$$2d_{eucl}(x,w) = d_{eucl}(x,w) + d_{eucl}(y,w) \geq d_{eucl}(x,y) \geq c \mu/4.$$
Now $w$ is at the intersection of three arcs $\alpha_{xy}$, $\alpha_{xz}$ and $\alpha_{yz}$. Here, $\alpha_{xy}$ is the set of points that are equidistant between $x$ and $y$ in the Euclidean metric, and  $\alpha_{xz}$ and $\alpha_{yz}$ are defined similarly. Let $x'$, $y'$, $z'$ be the perturbations of $x$, $y$, $z$.
The set of points in $M$ equidistant between $x'$ and $y'$ is, near $w$, a hyperplane. This intersects $[x,y,z]$ in an arc that runs very close to $\alpha_{xy}$. More precisely, these two arcs lie within distance at most $\delta c \mu$ of each other for some $\delta$, where $\delta \rightarrow 0$ as $c \rightarrow 0$ and $\lambda \rightarrow 0$. Hence, there is a point $w'$ in $N_{\delta c \mu}(w) \cap \partial M^{\textrm{fat}}_{\mu/2}$  that is equidistant in the hyperbolic metric between $x'$, $y'$ and $z'$. (To see that such a point $w'$ must exist, consider the locus of points in $N_{\delta c \mu}(w) \cap \partial M^{\textrm{fat}}_{\mu/2}$ equidistant between $x'$ and $y'$ and its intersection with the locus of points equidistant between $x'$ and $z'$.)
Then for any other point $p$ in $\calS \setminus \{ x, y, z \}$, its perturbation $p'$ satisfies 
\begin{align*}
&d_{hyp}(w',p')  \geq d_{hyp}(w,p) - d_{hyp}(w',w) - d_{hyp}(p',p) \\
&\quad \geq (1+\epsilon) d_{hyp}(w,x) - d_{hyp}(w',w) - d_{hyp}(p',p) \\
&\quad \geq (1+\epsilon) d_{hyp}(w',x) - (2 + \epsilon) d_{hyp}(w',w)  - d_{hyp}(p',p) \\
&\quad \geq (1+\epsilon) d_{hyp}(w',x')  - (1+\epsilon) d_{hyp}(x',x) - (2 + \epsilon) d_{hyp}(w',w) - d_{hyp}(p',p) \\
&\quad \geq (1+\epsilon) d_{hyp}(w',x') - (2+\epsilon) c \lambda \mu /1000 - (2 + \epsilon) \delta c \mu.
\end{align*}

The first, third and fourth inequalities are the triangle inequality.
The second inequality is a consequence of the hypothesis that $\calS$ is $\epsilon$-robust. 
Now $d_{hyp}(w',x')/(c\mu)$ is bounded away from zero, and so provided $c$ and $\lambda$ are sufficiently small,
we deduce that 
$$(2+\epsilon) c \lambda \mu /1000 + (2 + \epsilon) \delta c \mu < \epsilon d_{hyp}(w',x').$$
This implies that provided $c$ and $\lambda$ is sufficiently small, $w'$ is closer to $x', y', z'$ than it is to any other point of $\calS'$ in the hyperbolic metric. Thus, $w'$ generates a 1-cell containing $x'$, $y'$ and $z'$ in the Voronoi domain for $\calS'$. Dual to this is a face of the Delaunay triangulation for $\calS'$, which has vertices $x'$, $y'$ and $z'$. Hence, we have established that $x'$, $y'$ and $z'$ indeed form a face of the Delaunay triangulation of $M$.
\end{proof}

Under the above robustness assumption, we now know from the previous lemma that when the triangulation of $\partial M^{\textrm{fat}}_{\mu/2}$ is straightened in $M$, each face is simplicial
in the Delaunay triangulation of $M$. We also require that the straightened triangulation of $\partial M^{\textrm{fat}}_{\mu/2}$ forms a surface that is isotopic to $\partial M^{\textrm{fat}}_{\mu/2}$. This is achieved by the following lemma.

\begin{lemma}
\label{Lem:StraighteningIsIsotopic}
Let $\mathcal{T}$ be a triangulation of $\partial M^{\textrm{fat}}_{\mu/2}$ in which each 1-simplex is geodesic in the
Euclidean Riemannian metric with length at most $\mu/8$. Let $\mathcal{T}'$ be obtained by homotoping each edge and each face so that it is totally geodesic in $M$. Then $\mathcal{T}'$ forms a triangulation for a surface that is isotopic to $\partial M^{\textrm{fat}}_{\mu/2}$. Furthermore this surface lies in $M_{[\mu/4,\infty)}$.
\end{lemma}

\begin{proof}
Let $X$ be the subset of $M$ that is the union of the simplices in $\mathcal{T}'$. 
This lies in $M_{(0,\mu/2]}$. 
There is map
$$straight \colon \partial M^{\textrm{fat}}_{\mu/2} \rightarrow X$$ given by straightening.
Formally, each vertex of $\mathcal{T}$ is sent to the same vertex of $\mathcal{T}'$. Each 1-simplex
of $\mathcal{T}$ is parametrised at constant speed and is sent to a constant-speed geodesic in $M$.
Each 2-simplex of $\mathcal{T}$ is sent to its image simplex in $X$, using ordered barycentric coordinates,
much as in Section \ref{Sec:Barycentric}.

We claim that $X$ misses the core geodesics $\gamma$ in the solid toral components of $\mathrm{cl}(M - M^{\textrm{fat}}_{\mu/2})$.
Suppose that, on the contrary, there was a point of intersection with $\gamma$.
This would lie in the image of the interior of some 1-simplex of $\mathcal{T}$ or of some 2-simplex of $\mathcal{T}$.
It cannot lie in the image of some 1-simplex, since the endpoints of this 1-simplex would have
to have opposite meridional co-ordinates in $\partial M^{\textrm{fat}}_{\mu/2}$. Thus the 
edge of $\mathcal{T}$ joining them would have length at least $\mu/4$, by Lemma \ref{Lem:BoundaryHasControlledGeometry}. 
But we are
assuming that each edge length is at most $\mu/ 8$. The point of intersection
also cannot lie in the image of the interior of a 2-simplex. If it did, the simplex would intersect
$\gamma$ transversely, and so the boundary of this simplex would have come
from a path in $\mathcal{T}$ encircling the meridian. But the length of such a path
would be at most $3 (\mu / 8)$, whereas the length of any meridian is more than $\mu/2$.
This proves the claim.

There is a radial projection map from $\mathrm{cl}(M - M^{\textrm{fat}}_{\mu/2}) - \gamma$ to $\partial M^{\textrm{fat}}_{\mu/2}$, defined as follows.
For each point $x$ in $\mathrm{cl}(M - M^{\textrm{fat}}_{\mu/2}) - \gamma$,
$radial(x)$ is its closest point in $\partial M^{\textrm{fat}}_{\mu/2}$. This therefore restricts to a map
$$radial \colon X \rightarrow  \partial M^{\textrm{fat}}_{\mu/2}.$$

We claim that the restriction of $(radial )\circ (straight)$ to each simplex of $\mathcal{T}$ is an orientation-preserving diffeomorphism
onto its image. This is easiest to see in the case of simplices lying in the boundary tori of cusps (see Figure \ref{Fig:RadialStraight}). In that case,
each simplex is its own image under  $(radial )\circ (straight)$. 
For example, consider an edge of $\mathcal{T}$ lying in the boundary of a cusp. The inverse image of
this torus in $\mathbb{H}^3$ is a horosphere, which we may take to be the horizontal Euclidean plane $\{ (x,y,z) : z = 1 \}$ in the upper-half space
model for $\mathbb{H}^3$. The edge 
of $\mathcal{T}$ lifts to a straight arc in that plane. The map $straight$ send this arc to a hyperbolic geodesic with the same endpoints,
which is part of a Euclidean circle. Then the map $radial$ sends this Euclidean circular segment back to the original Euclidean geodesic.
The composition $(radial )\circ (straight)$ is not the identity on this edge of $\mathcal{T}$. This is because $straight$ sends the arc
along the hyperbolic geodesic with constant hyperbolic speed, and so when we compose with $radial$, it no longer has constant Euclidean
speed. Nevertheless it is clear that $(radial )\circ (straight)$ sends each simplex of $\mathcal{T}$ in the boundary of a cusp
back to itself via an orientation-preserving diffeomorphism. For simplices in components of $\partial M^{\textrm{fat}}_{\mu/2}$
bounding solid tori, the composition $(radial )\circ (straight)$ does not necessarily send each simplex back to itself. But it sends each
simplex onto its image via an orientation-preserving diffeomorphism.

\begin{figure}
  \includegraphics[width=0.7\textwidth]{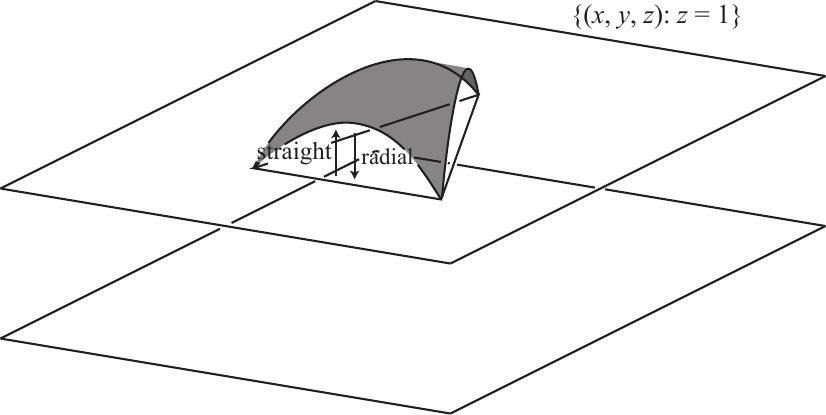}
  \caption{The maps {\it straight} and {\it radial} applied to a triangle in a triangulation of a cusp torus.}
  \label{Fig:RadialStraight}
\end{figure}

As a result of this claim, the composition $(radial )\circ (straight)$ is a branched covering map
from each component of  $\partial M^{\textrm{fat}}_{\mu/2}$ to itself. (In fact, it is easy that it is
actually a local homeomorphism of $\partial M^{\textrm{fat}}_{\mu/2}$,
by examining its action in a small neighbourhood of any vertex of $\mathcal{T}$,
but we do not need to follow that line of argument.) 
The composition $(radial )\circ (straight)$ is homotopic to the identity (via a straight line homotopy in the Euclidean tori $\partial M^{\textrm{fat}}_{\mu/2}$),
and so it has degree $1$. Any degree 1 branched cover from a torus to itself is a homeomorphism.
In particular, $straight$ is injective. Because $\partial M^{\textrm{fat}}_{\mu/2}$ is compact and $X$ is Hausdorff, $straight$ is therefore
a homeomorphism onto its image. Thus, its image $X$ is indeed a surface, and $\mathcal{T}'$
forms a triangulation for it.

Note that $X$ is homotopic to $\partial M^{\textrm{fat}}_{\mu/2}$ via a straight line homotopy in $\mathrm{cl}(M - M^{\textrm{fat}}_{\mu/2}) - \gamma$.
Both $X$ and $\partial M^{\textrm{fat}}_{\mu/2}$ are incompressible and embedded in 
$\mathrm{cl}(M - M^{\textrm{fat}}_{\mu/2}) - \gamma$, and such surfaces
are isotopic if and only if they are homotopic.

The last thing to prove is that each simplex of $X$ lies in $M_{[\mu/4,\infty)}$. Consider any point $x$ in $X$.
This lies in a totally geodesic simplex of $X$. The distance from $x$ to any of the vertices of this simplex is at most the
maximal length of the edges of the simplex, which is at most $\mu/8$. So, there is a path $p$ of length at most $\mu/8$
joining $x$ to a point $x'$ in $\partial M^{\textrm{fat}}_{\mu/2} \subseteq M_{[\mu/2,\infty)}$. If there were a homotopically non-trivial 
loop based at $x$ with length less than $\mu/4$, we could use this to build a homotopically non-trivial loop based at $x'$, by travelling
backwards along $p$ to $x$, then along the loop based at $x$, then forwards along $p$. This loop would have length
less than $\mu/8+\mu/4+\mu/8 = \mu/2$, which contradicts the fact that $x'$ lies in $M_{[\mu/2,\infty)}$.

\end{proof}

We will actually need a strengthening of the above lemma, where the vertices of the triangulation
$\calT$ are perturbed a little, as follows.

\begin{lemma}
\label{Lem:StraighteningIsIsotopicVariation}
Let $\mu$ be any real number less than the 3-dimensional Margulis constant, and let $\tilde a, \tilde b, \tilde \theta$ be positive real numbers
where $0 < \tilde a < \tilde b < 1/10$ and $\tilde \theta \in (0,\pi/2)$. 
Let $\mathcal{T}$ be a triangulation of $\partial M^{\textrm{fat}}_{\mu/2}$ in which each 1-simplex is geodesic in the
Euclidean Riemannian metric, has edge lengths between $\tilde a \mu$ and $\tilde b \mu$, and
interior angles at least $\tilde \theta$. Let $\mathcal{T}'$ be obtained by perturbing the
vertices of $\mathcal{T}$ in $M$ by distance at most $(5/12) \min \{ \tilde a \mu/100, \tilde a \mu (\tilde \theta/ 2\pi) \}$ and then
homotoping each edge and each face so that it is totally geodesic in $M$. 
Then $\mathcal{T}'$ forms a triangulation for a surface that is isotopic to $\partial M^{\textrm{fat}}_{\mu/2}$.
Furthermore this surface lies in $M_{[\mu/4,\infty)}$.
\end{lemma}

In the proof, we will need the following elementary fact about the perturbation of Euclidean triangles.

\begin{lemma}
\label{Lem:PerturbTriangle}
Let $T = [x,y,z]$ be a Euclidean triangle, with side lengths between $A$ and $B$, and interior angles at least $\tilde \theta \in (0,\pi/2)$.
Then, for any Euclidean triangle $\tilde T$ obtained from $T$ by perturbing each vertex a distance less than $(A/2)\sin(\tilde \theta/2)$,
the orientations of $\tilde T$ and $T$ are equal, i.e. the cyclic order of vertices remains unchanged under the perturbation. 
\end{lemma}

\begin{proof}
Let $\epsilon = (1/2)\sin(\tilde \theta/2)$.
Let $\tilde x$, $\tilde y$ and $\tilde z$ be the vertices of $\tilde T$ obtained by perturbing $x$, $y$ and $z$ respectively, with
each vertex moving less than $\epsilon A$. 
Translate $\tilde T$ so that afterwards $\tilde x$ equals $x$. Then, after this translation, $\tilde y$ and $\tilde z$ lie at a distance at most $2 \epsilon A$
from $y$ and $z$. It suffices to show that the angle between $xy$ and $x \tilde y$ is less than $\tilde \theta/2$
and the angle between $xz$ and $x \tilde z$ is less than $\tilde \theta/2$.
This is because the angle between $xy$ and $xz$ is, by assumption, at least $\tilde \theta$ and at most $\pi - 2 \tilde \theta$. Hence, by shifting
$y$ and $z$, the angle between $x \tilde y$ and $x \tilde z$ remains positive. It also remains at most $\pi$.

We consider the triangle $[x,y,\tilde y]$. (The argument for $[x,z,\tilde z]$ is entirely analogous.)
Let $X$, $Y$ and $\tilde Y$ be the angles at these three vertices.
An application of the sine rule at $x$ gives 
$$\frac{\sin X}{|y \tilde y|} = \frac{\sin \tilde Y}{|x y|}$$
and so
$$\sin X = \frac{|y \tilde y| \sin \tilde Y}{|x y|} \leq \frac{2 \epsilon A}{A} = 2 \epsilon.$$
Since $\epsilon = (1/2)\sin(\tilde \theta/2)$, we obtain the lemma.

\end{proof}

\begin{proof}[Proof of Lemma \ref{Lem:StraighteningIsIsotopicVariation}]

There is a map $f \colon \partial M^{\textrm{fat}}_{\mu/2} \rightarrow M$ that is defined in an analogous way to
$straight$ but first performing the given perturbation of the vertices: we perturb the vertices
of $\mathcal{T}$ to the corresponding vertices of $\mathcal{T}'$ and then straighten the
edges and faces. The image of $f$ lies in $V = N_{(5/12)\min \{\tilde a \mu/100, \tilde a \mu (\tilde \theta/ 2\pi) \} }(M - M^{\textrm{fat}}_{\mu/2})$. 
Moreover it misses the union of the core geodesics, denoted by $\gamma$, of the solid toral components of $V$, by the same argument as in the
proof of Lemma \ref{Lem:StraighteningIsIsotopic}, since the perturbation distance for the vertices is small enough. Let $X$ be the image of $f$.

There is also a map $radial_+ \colon V - \gamma  \rightarrow \partial M^{\textrm{fat}}_{\mu/2}$ given
by closest-point projection. This restricts to 
$radial_+ \colon X \rightarrow \partial M^{\textrm{fat}}_{\mu/2}$. As in the proof of Lemma \ref{Lem:StraighteningIsIsotopic},
the composition $(radial_+ \circ f)$ might not send edges to edges, or faces to faces. However, this time it also
might not send vertices to vertices, because of our initial perturbation. But it moves each vertex $v$
a Euclidean distance at most $ \min \{ \tilde a \mu/100, \tilde a \mu (\tilde \theta/ 2\pi) \}$, for the following reason.

The initial perturbation moves $v$ to a point $v'$ which is at most 
$$D = (5/12) \min \{ \tilde a \mu/100, \tilde a \mu (\tilde \theta/ 2\pi) \}$$ 
away. The map
$radial_+$ sends $v'$ to the nearest point $v''$ in $\partial M^{\textrm{fat}}_{\mu/2}$. Thus,
$v''$ and $v'$ are also at most $D$ apart. Hence, the hyperbolic distance between
$v$ and $v''$ is at most $2D$. Let $d_{eucl}$ be 
the Euclidean distance  in $\partial M^{\textrm{fat}}_{\mu/2}$
between $v$ and $v''$. Since $2D < \mu/100$, we deduce 
from Lemma \ref{Lem:DistanceComparison} that
$$d_{eucl} \leq (1.12)2D < (12/5)D,$$
as required.

Consider any face of $\mathcal{T}$, which is a Euclidean triangle $\Delta$. Its vertices are moved by
$(radial_+ \circ f)$, each by a distance at most $(12/5)D$. Now, 
$$(12/5)D \leq \tilde a \mu (\tilde \theta/ 2\pi) \leq (\tilde a \mu/2) \sin (\tilde \theta /2).$$ 
So by Lemma \ref{Lem:PerturbTriangle}, the Euclidean triangle spanned by these
vertices has the same orientation as  $\Delta$. As in the proof of Lemma \ref{Lem:StraighteningIsIsotopic},
the restriction of $(radial_+ \circ f)$ to $\Delta$ is a diffeomorphism onto its image
and it is orientation-preserving. So the map $(radial_+ \circ f)$
is a branched covering map from the torus to itself that is also homotopic to the identity,
and hence a homeomorphism. Therefore, $f$ is a homeomorphism onto its image,
as required. Again, this image surface $X$ is homotopic in $\mathrm{cl}(M - M^{\textrm{fat}}_{\mu/2}) -  \gamma$ to $\partial M^{\textrm{fat}}_{\mu/2}$,
and hence they are isotopic.

Finally, the  surface $X$ lies in $M_{[\mu/4,\infty)}$, for the following reason. 
The distance from any point $x$ of $X$ to $\partial M^{\textrm{fat}}_{\mu/2}$ is at most $\tilde b\mu + 3(\tilde a \mu/100)$. 
This is because each vertex of $X$ lies within $\tilde a \mu/100$ of $\partial M^{\textrm{fat}}_{\mu/2}$
and each point in $X$ has distance at most $\tilde b\mu + 2(\tilde a \mu/100)$ from a vertex.
If there were a homotopically non-trivial 
loop based at $x$ with length less than $\mu/4$, we could use this to build a homotopically non-trivial loop based at a point $x'$
in $\partial M^{\textrm{fat}}_{\mu/2}$ with length less than
$2(\tilde b\mu + 3(\tilde a \mu/100)) + \mu/4 < \mu/2$.
The latter inequality holds because $0 < \tilde a < \tilde b < 1/10$.
The existence of this loop based at $x'$ contradicts the fact that $x'$ lies in $M_{[\mu/2,\infty)}$.

\end{proof}

There is one final result that we will need. It is an analogue of Theorem \ref{Thm:ExtensionBreslin}, but for straight triangulations
of Euclidean surfaces. It is essentially proved in \cite{Boissonnat2} (see Theorem 4.1 \cite{Boissonnat2} which deals with
triangulations of $\mathbb{R}^n$).

\begin{theorem}
\label{Thm:ThickTriangTorus}
There are real numbers $0 < \tilde a < \tilde b < 1/10$, $\tilde \theta \in (0,\pi/2)$ and $\epsilon \in (0,1/10)$
with the following property. Let $\mu$ be any positive real number, let $c \in (0, 1/200)$, and let $T$ be a 
Euclidean torus with injectivity radius at least $\mu/4$. Let $\mathcal{S}_\partial$ be a maximal collection of points in $T$, no two of which
are closer than $c\mu$. Then there is a perturbation $\tilde{\mathcal{S}}_\partial$ of $\mathcal{S}_\partial$ in $T$, moving each point at most $c\mu/1000$, 
so that 
\begin{enumerate}
\item the Delaunay triangulation arising from $\tilde{\mathcal{S}}_\partial$ has edge lengths between $\tilde a \mu$ and $\tilde b \mu$ and interior angles at least $\tilde \theta$;
\item for each point in $T$ that has exactly three closest points in $\tilde{\calS}_\partial$, the next closest point in $\tilde{\calS}_\partial$ is least $(1+3\epsilon)$ times further away.
\end{enumerate}
\end{theorem}

In the closed case, the proof of Theorem \ref{Thm:ExtensionBreslin} proceeded by picking a finite set of points $\mathcal{S}$ in $M$ and then making a good perturbation, giving a set $\mathcal{S}'$ that forms the vertices of the required thick triangulation. In the case where $M^\mathrm{fat}_{[\mu/2, \infty)}$ has non-empty boundary, we need somewhat more control over $\mathcal{S}$, as follows.

\begin{proposition}
\label{Prop:BoundaryIsSimplicial}
Let $\tilde a$, $\tilde b$, $\tilde \theta$ and $\epsilon$ be the constants provided by Theorem \ref{Thm:ThickTriangTorus}.
Let $c$ and $\lambda$ be the constants provided by Lemma 
\ref{Lem:RobustImpliesSimplicial}.
There is a finite set $\calS$ of points in $M$ with the following properties:
\begin{enumerate}
\item any two points in $\calS \cap \mathrm{int}(M^\mathrm{fat}_{[\mu/2, \infty)})$ are at least $\mu/90$ apart, 
\item any $p \in M_{[\mu/2, \infty)}$ has distance less than $\mu/90$ from some point in $\calS$;
\item for every point $p$ in $M$, the ball of radius $\mathrm{inj}_p(M)/5$ about $p$ has at least one point of $\calS$ in its interior.
\end{enumerate}
Furthermore, for any perturbation $\mathcal{S}'$ of $\mathcal{S}$ moving each point at most
$$D' = \min \{ (5/12) \tilde a \mu/100, (5/12) \tilde a \mu (\tilde \theta/ 2\pi), c \lambda \mu/1000 \}$$
and where the points are the vertices of a Delaunay triangulation of $M$, this
Delaunay triangulation includes a simplicial subset $\mathcal{Y}$ isotopic to $\partial M^{\textrm{fat}}_{\mu/2}$ 
and that lies in $M_{[\mu/4,\infty)}$.
\end{proposition}

\begin{proof}
Let $\mathcal{S}_\partial$ be a maximal collection of points on $\partial M^{\textrm{fat}}_{\mu/2}$, no two of which are closer than $c\mu$. 
By Theorem \ref{Thm:ThickTriangTorus}, we may perturb $\mathcal{S}_\partial$ to $\tilde{\mathcal{S}}_\partial$ in $\partial M^{\textrm{fat}}_{\mu/2}$, 
moving each point at most $c\mu/1000$, so that the resulting Delaunay triangulation 
$\mathcal{T}$ of $\partial M^{\textrm{fat}}_{\mu/2}$ with its induced Riemannian metric has edge lengths between $\tilde a \mu$ and $\tilde b \mu$, and
interior angles at least $\tilde \theta$. Furthermore, for each point in $\partial M^{\textrm{fat}}_{\mu/2}$ that has exactly three closest points in $\tilde \calS_\partial$ (in the Euclidean metric), the next closest point in $\tilde \calS_\partial$ is least $(1+3\epsilon)$ times further away (in the Euclidean metric). 
Note that after this perturbation, the Euclidean balls about these points
of radius $(1001/1000)c\mu$ cover $\partial M^{\textrm{fat}}_{\mu/2}$, and any two points are at least $(998/1000)c\mu$ apart.
So by Lemma \ref{Lem:DistanceComparison}, the hyperbolic balls about these points of radius $(1001/1000)c\mu$
cover $\partial M^{\textrm{fat}}_{\mu/2}$, and any two points are at least $(89/100)c\mu$ apart in the hyperbolic metric by Lemma \ref{Lem:DistanceComparison}.
By Lemma \ref{Lem:StraighteningIsIsotopicVariation}, when we perturb the
vertices of $\mathcal{T}$ in $M$ by distance at most $(5/12) \min \{ \tilde a \mu/100, \tilde a \mu (\tilde \theta/ 2\pi) \}$  
and then straighten, the result is a triangulation of a surface $X$ that is isotopic to $\partial M^{\textrm{fat}}_{\mu/2}$ and lying in $M_{[\mu/4,\infty)}$. This will be the simplicial subset $\mathcal{Y}$ required by the proposition. 

We need to define an extension of $\tilde{\mathcal{S}}_\partial$ to a set of points $\mathcal{S}$ in $M$. 
Define $\mathcal{S} \cap M^{\textrm{fat}}_{\mu/2}$ as follows. Take it to be a maximal set having the following properties:
\begin{enumerate}
\item $\mathcal{S} \cap M^{\textrm{fat}}_{\mu/2}$ contains $\tilde{\mathcal{S}}_\partial$;
\item for any $p \in (\mathcal{S} \cap M^{\textrm{fat}}_{\mu/2}) \setminus \tilde{\mathcal{S}}_\partial$ and $q \in \mathcal{S} \cap M^{\textrm{fat}}_{\mu/2}$, the distance between $p$ and $q$ is at least 
$\mu/90$.
\end{enumerate}
Then extend $\calS \cap M^{\textrm{fat}}_{\mu/2}$ to a set $\calS$ by adding points in $M - M^{\textrm{fat}}_{\mu/2}$, so that (1) and (2) still hold
and in addition:
\begin{enumerate}
\item[(3)] for any $p \in \mathcal{S} - M^{\textrm{fat}}_{\mu/2}$ and $q \in \mathcal{S} \cap M^{\textrm{fat}}_{\mu/2}$, the distance between $p$ and $q$ is at least $\mu/90$;
\item[(4)] for any $p \in M$, the ball $B(p, \mathrm{inj}(M,p)/5)$ centred at $p$ with radius $\mathrm{inj}(M,p)/5$ has a point of $\calS$ in its interior;
\item[(5)] for each $p$ in $\mathcal{S} - M^{\textrm{fat}}_{\mu/2}$, the distance from $p$ to any other point in $\mathcal{S} - M^{\textrm{fat}}_{\mu/2}$
is at least $\mathrm{inj}(M,p)/10$.
\end{enumerate}
This guarantees that the distance between any two points in $\mathcal{S} \cap M^{\textrm{fat}}_{\mu/2}$ is at least $(89/100)c\mu$. This is true by (2) when at least one of the points is in $\mathcal{S} \setminus \tilde{\mathcal{S}}_\partial$ (due to $c$ being less than $1/200$ in Lemma 
\ref{Lem:RobustImpliesSimplicial}), and we observed this above when both points are in $\tilde{\mathcal{S}}_\partial$. We also note that every point in $M^{\textrm{fat}}_{\mu/2}$ has distance less than 
$\mu/90$ from some point in $\mathcal{S}$. For otherwise we could add this point to $\calS \cap M^{\textrm{fat}}_{\mu/2}$ contradicting its maximality. Hence, the balls of radius 
$\mu/90$ about $\calS$ cover $M^{\textrm{fat}}_{\mu/2}$.

 We claim that $\calS$ is $\epsilon$-robust. Consider any three points $x$, $y$, $z$ in $\partial M^{\textrm{fat}}_{\mu/2}$ that span a triangle in $\calT$. Let $w$ be their Euclidean circumcentre. The Euclidean distance from $w$ to any of $x$, $y$, $z$ is at most $(1001/1000) c \mu$. So the hyperbolic distance from 
$w$ to any of $x$, $y$, $z$ is also at most $(1001/1000) c \mu \leq \mu/199.$
 By Theorem \ref{Thm:ThickTriangTorus}, the next closest point in $\partial M^{\textrm{fat}}_{\mu/2}$ to $w$ is at least $(1 + 3 \epsilon)$ further away in the Euclidean metric.  Hence, it is at least $(1 + \epsilon)$ times further away in the hyperbolic metric, by Lemma \ref{Lem:DistanceComparison},
as required. Moreover, by (2) and (3) above, any point in $\calS \setminus \partial M^{\textrm{fat}}_{\mu/2}$ has hyperbolic distance at least $\mu/90$
from $x$, and hence hyperbolic distance at least $(\mu/90) - (\mu/199)$ from $w$. This verifies the claim.

Now consider any perturbation $\mathcal{S}'$ of $\mathcal{S}$, moving each point a distance at most $D'$,
and resulting in a set that is the vertex set of a Delaunay triangulation of $M$. Let $\mathcal{S}'_\partial$ be the subset of $\mathcal{S}'$ arising from $\tilde{\mathcal{S}}_\partial$.

 By Lemma \ref{Lem:RobustImpliesSimplicial}, $X$ is simplicial in the Delaunay triangulation of $M$ with vertex set $\calS$, as required.
 We have already verified above that $X$ is a surface isotopic to $\partial M^{\textrm{fat}}_{\mu/2}$ and lying in $M_{[\mu/4,\infty)}$.\end{proof}
 
Note that each point in $M^{\textrm{fat}}_{\mu/2}$ has distance at most $(\mu/90) + 2D' < \mu/80$ from some point in $\calS'$. In particular, this is true
of the circumcentre of any tetrahedron in the Delaunay triangulation with vertex set $\calS'$. So each edge in this triangulation has length less than $\mu/40$.

The rest of the proof of Theorem \ref{Thm:ExtensionBreslin} proceeds as in the closed case. We pick a total ordering on the set $\mathcal{S}$ from Proposition \ref{Prop:BoundaryIsSimplicial},
and we perturb these points one point at a time. Let $\mathcal{S}_i$ be the result of perturbing the first $i$ points, and let $p_{i+1}$ be the
next point to perturb. Let $\mathcal{U}_i$ be the set of triples $\{ q,r,s \} \in \mathcal{S}_i \setminus \{ p_{i+1} \}$ such that $[p_{i+1}', q', r', s']$ is a tetrahedron in the Delaunay triangulation on some good perturbation of $\mathcal{S}$. 
The distance that we perturb each point is at most $D'$,
and we choose the perturbation so that it avoids
the sliver region of each triangle in $\mathcal{U}_i$. As previously, if $\sigma$ is sufficiently small, we can ensure that
there is such a perturbation of $p_i$. Repeating this for all points of $\mathcal{S}$ gives a perturbation $\mathcal{S}'$.
By Proposition \ref{Prop:BoundaryIsSimplicial}, the Delaunay triangulation with vertex set $\mathcal{S}'$ forms a triangulation
a subset of $M$ isotopic to $M^{\textrm{fat}}_{\mu/2}$ and lying in $M_{[\mu/4,\infty)}$.
Since we have avoided all sliver regions, it is a $(\theta, a\mu,b\mu)$-thick triangulation. 
This completes the proof of Theorem
\ref{Thm:ExtensionBreslin}.

\medskip 
\noindent
Marc Lackenby\\
Mathematical Institute\\
University of Oxford\\
Andrew Wiles Building\\
Radcliffe Observatory Quarter\\
Woodstock Road, Oxford, OX2 6GG, UK \\
marc.lackenby@maths.ox.ac.uk

\medskip
\noindent
Anastasiia Tsvietkova\\
Department of Mathematics and Computer Science\\
Rutgers University-Newark \\
101 Warren Street, Newark, NJ  07102, USA\\
a.tsviet@rutgers.edu

\end{document}